\documentclass[11pt]{article}

\usepackage{latexsym,amsmath,amsfonts,amssymb,dsfont,wasysym}
\usepackage{amsthm}
\usepackage{enumerate}
\usepackage{epsfig}
\usepackage{verbatim}
\usepackage{xcolor}
\usepackage{enumitem}
\usepackage{mathrsfs}
\usepackage{hyperref,cite}
\usepackage[title,titletoc,toc]{appendix}
\usepackage{tikz}
\usepackage{graphicx,subcaption}

\def\td{\mathrm{d}}

\def\diam{\mathrm{diam}} 
 \def\det{\mathrm{det}}
\def\supp{\mathrm{supp}} \def\dist{\mathrm{dist}}
\def\spa{\mathrm{span}} \def\diag{\mathrm{diag}}
\def\sign{\mathrm{sgn}} \def\tD{\mathrm{D}}
 
\def\ex{\mathds{E}} \def\eps{\varepsilon}
\def\prb{\mathds{P}} \def\qrb{\mathds{Q}}
 \def\wass{\mathds{W}}
\def\NN{\mathbb{N}} \def\ZZ{\mathbb{Z}}
 \def\RR{\mathbb{R}}

\def\wt{\widetilde} \def\wh{\widehat}
\def\lb{\left(} \def\rb{\right)}
\def\lv{\left|} \def\rv{\right|}
\def\lvv{\left\|} \def\rvv{\right\|}
 
\def\lq{\leqslant} \def\gq{\geqslant}

 \def\N{\mathcal{N}}
\newcommand\num{\stepcounter{equation}\tag{\theequation}}

\newtheorem*{theorem*}{Theorem}
\newtheorem{theorem}{Theorem}
\newtheorem{lemma}[theorem]{Lemma}
\newtheorem{corollary}[theorem]{Corollary}

\newtheorem{remark}[theorem]{Remark}

\usepackage[top=3.0cm, bottom=3.0cm, left=3.0cm, right=3.0cm]{geometry}
\makeatletter \@addtoreset{equation}{section}
\renewcommand{\thesection}{\arabic{section}}
\renewcommand{\theequation}{\thesection.\arabic{equation}}

\title{\textsc{A Coupling for Triple Stochastic Integrals}}

\author{X\={\i}l\'{\i}ng Zh\={a}ng
        \thanks{School of Mathematics, The University of Edinburgh. e-mail: xiling.zhang@ed.ac.uk}}

\begin{document}
\maketitle
\begin{abstract}
This article presents a coupling approach for the approximation of iterated stochastic integrals of length three. The generation of such integrals is the central problem of higher-order pathwise approximations for SDEs, which still lacks a satisfactory answer due to the restriction of dimensionality and computational load. Here we start from the Fourier representation of the triple stochastic integral and investigate the global behaviour of the joint density of the representation. Finally in the main result we give a coupling in the quadratic Vaserstein distance.
\end{abstract}

\section{Introduction}

Let\footnote{Throughout the paper $\ZZ^+$ denotes the set of positive integers and $\NN$ denotes the set of natural numbers $\ZZ^+\cup\{0\}$.} $d,q\in\ZZ^+$ and $(\Omega,\mathscr{F},\prb)$ be a complete probability space equipped with a right-continuous filtration $\mathbb{F}=\{\mathscr{F}_t\}_{t\gq0}$. Consider an $\RR^d$-valued autonomous stochastic differential equation driven by a $q$-dimensional $\mathbb{F}$-Wiener martingale $W$:
\begin{equation}\label{sde}
x_t=x_0+\int_0^tb(x_s)\td s+\int_0^t\sigma(x_s)\td W_s.
\end{equation}
Assume that the coefficients $b:\RR^d\to\RR^d$ and $\sigma:\RR^d\to\RR^{d\times q}$ are sufficiently smooth. It is well-known that one can derive numerical schemes that converge in the strong $L^p$ sense of order greater than $1/2$ from stochastic Taylor expansions, as is shown in \cite{kloeden1995NumSolStoDifEqu}. For example, by applying It\^o's formula to $b$ and $\sigma$, one obtains the It\^o-Taylor expansion of length $2$: for each component $i=1,\cdots,d$ on the interval $[s,t]$,
\begin{align}
x_t^i=&x_s^i+b_i(x_s)(t-s)+\sum_{j=1}^q\sigma_{ij}(x_s)(W_t^j-W_s^j)\nonumber\\
&+\int_s^t\int_s^r\mathcal{L}b_i(x_u)\td u\td r+\sum_{j=1}^q\int_s^t\int_s^r\sum_{k=1}^d\sigma_{kj}(x_u)\partial_kb_i(x_u)\td W_u^j\td r\nonumber\\
&+\sum_{j=1}^q\int_s^t\int_s^r\mathcal{L}\sigma_{ij}(x_u)\td u\td W_r^j+\sum_{j,k=1}^q\int_s^t\int_s^r\sum_{l=1}^d\sigma_{lk}(x_u)\partial_l\sigma_{ij}(x_u)\td W_u^k\td W_r^j,\label{ito_taylor_2}
\end{align}
where $\partial_k$ is the partial derivative w.r.t. the $k$-th coordinate. The last term in \eqref{ito_taylor_2} involves an iterated stochastic integral, and it gives rise to Milstein's method: for each component $i=1,\cdots,d$,
\begin{equation}\label{milstein}
X_{k+1}^i=X_k^i+b_i(X_k)h+\lb\sum_{j=1}^q\sigma_{ij}(X_k)\Delta W_{k+1}^j+\sum_{j,l=1}^q\varsigma_{ijl}(X_k)A_k(j,l)\rb,
\end{equation}
where $h\in(0,1)$ is the step size, $\Delta W^j_{k+1}=W^j_{t_{k+1}}-W^j_{t_k},~\varsigma_{ijl}(x):=\sum_{m=1}^d\sigma_{mj}(x)\partial_m\sigma_{il}(x)$ and
\begin{equation*}
A_k(j,l):=\int_{t_k}^{t_{k+1}}(W_t^j-W_{t_k}^j)\td W_t^l.
\end{equation*}
The scheme \eqref{milstein} has strong-$L^2$ convergence rate $O(h)$ according to Kloeden and Platen \cite{kloeden1995NumSolStoDifEqu} (Section 10.3), but the problem lies in the generation of the double integral $I_{jl}=\int_0^hW_t^j\td W_t^l$, which is non-trivial for $q\gq2$.

As mentioned by Wiktorsson \cite{wiktorsson2001JoiChaFunSimSimIteItIntMulIndBroMot} and Davie \cite{davie2014Patappstodifequusicou} (Section 2), if the diffusion matrix satisfies the commutativity condition $\varsigma_{ijl}(x)=\varsigma_{ilj}(x)$ for all $x\in\RR^d$ and all $i=1,\cdots,d,~j,l=1,\cdots,q$, one only needs to generate the Wiener increments $\Delta W_{k+1}$ to achieve the order-$1$ convergence. But this is not always the case: using only the Wiener increments $\Delta W_{k+1}$ to implement a numerical method will, in general, result in a convergence rate no more than $O(h^{1/2})$, according to \cite{clark1978MaxRatConDisAppStoDifEqu}.

One attempt to generate the double integral $I_{jl}$ was made by Lyons and Gaines \cite{lyons1994RanGenStoAreInt}, but their method only works for $q=2$. Recently a strong result for any dimension has been proved by Davie \cite{davie2014Patappstodifequusicou} (Theorem 4) under the condition that the diffusion matrix $\sigma$ has rank $d$ everywhere, and it provides a way to approximate the SDE up to an arbitrary order. This is a significant improvement concerning higher-order approximations. The idea is that, rather than generating the double integrals at each step $k$, one approximates the quantity inside the big parentheses in \eqref{milstein} as a whole. This is a completely different approach than the usual ones, as Davie's arguments are based on the coupling method, quantifying the strong-$L^p$ convergence in terms of the Vaserstein\footnote{Also spelt as ``Wasserstein".} metrics.

\paragraph{The coupling method.} For probability measures $\prb,\qrb$ on $\RR^q$ and $p\gq1$, the \textit{Vaserstein $p$-distance} is defined by
\begin{equation*}
\wass_p(\prb,\qrb):=\inf_{\pi\in\Pi(\prb,\qrb)}\lb\int_{\RR^q\times\RR^q}|x-y|^p\pi(\td x,\td y)\rb^{1/p},
\end{equation*}
where $\Pi(\prb,\qrb)$ is the set of all joint probability measures on $\RR^q\times\RR^q$ with marginal laws $\prb$ and $\qrb$. In general $\prb$ and $\qrb$ need not be defined on the same probability space, but this definition is enough for the purpose of this article. The notation $\wass_p(X,Y)$ will not cause any confusion for random variables $X$ and $Y$ having laws $\prb$ and $\qrb$, respectively. If one can show a bound for the distance between the two laws, we then say there is a \textit{coupling} between $X$ and $Y$ (or $\prb$ and $\qrb$).

The significance of using the Vaserstein distances instead of other ones is that, when generating numerical schemes for an SDE, the convergence in the Vaserstein-type distance $\wass_{p,\infty}$ (replacing $|x-y|^p$ in the definition above by $\max_k|x_k-y_k|^p$) is equivalent to the usual strong $L^p$-convergence, for the purpose of simulation at least. To see this, suppose we have found a coupling between the grid points of the solution $x=\{x_{t_k}\}_k$ and a numerical scheme $X=\{X_k\}_k$ with $\wass_{p,\infty}(x,X)\lq Ch^\gamma$ for some $\gamma>0$. Then by definition, $\forall\eps>0$ there is a random vector $Y^\eps$ on the same probability space as the solution $x$, having the same distribution as $X$, s.t. $(\ex\max_k|x_{t_k}-Y_k|^p)^{1/p}\lq\wass_{p,\infty}(x,X)+\eps$. Choose $\eps=h^\gamma$ and in practice one generates $Y$ instead of $X$ to approximate $x$. The reader is referred to Section 12 in \cite{davie2014Patappstodifequusicou} for a detailed discussion on the contexts where such a substitution holds or fails.

Although there is no general formulas for the quantity $\wass_p(\prb,\qrb)$, if $\prb$ and $\qrb$ have densities $f$ and $g$, respectively, then there is the elementary and yet important inequality
\begin{equation}\label{crude_bound_1}
\wass_p(\prb,\mathds{Q})\lq C_p\lb\int_{\RR^q}|x|^p|f(x)-g(x)|\td x\rb^{1/p},
\end{equation}
for all $p\gq1$, as a variant of Proposition 7.10 in \cite{villani2003TopOptTra}. This inequality serves as a main tool to give an $\wass_2$-estimate in \cite{davie2014Patappstodifequusicou} and \cite{davie2015Polpernordis}, and will be used for all the main result in this article.

The more difficult situation is that $\sigma$ has rank less than $d$, which could well happen. In Section 9 in \cite{davie2014Patappstodifequusicou} a different approach based on the Fourier expansion introduced in Section 5.8 in \cite{kloeden1995NumSolStoDifEqu} is proposed, giving a coupling for the double integral $I_{jl}$. The motivation of this article is to provide a feasible approximation for SDEs of a higher order. For the equation \eqref{sde} on the interval $[0,T]$, by applying It\^o's formula again to the term $\sigma_{kl}(X_u)\partial_k\sigma_{ij}(X_u)$ in \eqref{ito_taylor_2}, one obtains, for each component $i=1,\cdots,d$ on the interval $[s,t]$,
\begin{align*}
X_t^i=&X_s^i+b_i(X_s)(t-s)+\sigma_{ij}(X_s)(W_t^j-W_s^j)+\sigma_{kl}(X_s)\partial_k\sigma_{ij}(X_s)\int_s^t\int_s^r\td W_u^l\td W_r^j\\
&+\int_s^t\int_s^r\mathcal{L}b_i(X_u)\td u\td r+\int_s^t\int_s^r\sigma_{kl}(X_u)\partial_kb_i(X_u)\td W_u^l\td r\\
&+\int_s^t\int_s^r\mathcal{L}\sigma_{ij}(X_u)\td u\td W_r^j+\int_s^t\int_s^r\int_s^u\mathcal{L}\lb\sigma_{kl}(X_v)\partial_k\sigma_{ij}(X_v)\rb\td v\td W_u^l\td W_r^j\\
&+\int_s^t\int_s^r\int_s^u\partial_m\lb\sigma_{kl}(X_v)\partial_k\sigma_{ij}(X_v)\rb\sigma_{mn}(X_v)\td W_v^n\td W_u^l\td W_r^j,
\end{align*}
where the summation signs over repeated indices are omitted. From this expression one can obtain a suitable numerical scheme (formula (10.4.6) in \cite{kloeden1995NumSolStoDifEqu}) with strong convergence order $O(h^{3/2})$. Just as the Milstein scheme, the crucial ingredient to achieve such a higher-order convergence is the generation of the triple integrals
\begin{equation*}
I_{jkl}(s,t):=\int_s^t\int_s^r\int_s^u\td W_v^j\td W_u^k\td W_r^l,
\end{equation*}
for indices $(j,k,l)\in\{1,\cdots,q\}^3$.

Similar to the way the double stochastic integral is treated in \cite{davie2014Patappstodifequusicou}, one would expect the same method to be extended to treat triple integrals. For the simplicity of formulation, the Stratonovich triple integral $I^\circ_{jkl}(s,t):=\int_s^t\int_s^r\int_s^u\td W_v^j\circ\td W_u^k\circ\td W_r^l$ will be considered instead of the It\^o version, since the Fourier representation of the former has a relatively simpler form. This is due to the fact that the product of two Stratonovich integrals is a shuffle product - see Proposition 2.2 in \cite{gaines1994AlgIteStoInt}. In other words, an iterated Stratonovich integral of longer length can be represented by shorter ones in a much simpler way compared its It\^o counterpart.

\paragraph{The double integral case.} The goal of this paper is to find a random variable $\bar{I}_{jkl}$ whose law is close to that of $I^\circ_{jkl}$ in the Vaserstein distance, which in turn gives a feasible $O(h^{3/2})$-approximation for the SDE \eqref{sde}. In order to have a better understanding of the method let us briefly review Davie's Fourier method (Section 9 in \cite{davie2014Patappstodifequusicou}). Consider the interval $[0,1]$ for simplicity. According to \cite{kloeden1995NumSolStoDifEqu} (Section 5.8), the Brownian bridge process $W_t-tW_1$ has Fourier expansion
\begin{equation}\label{bridge}
W_t^j-tW_1^j=\frac{1}{2\sqrt{2}\pi}x_{j0}+\frac{1}{\sqrt{2}\pi}\sum_{r=1}^\infty x_{jr}\cos(2\pi rt)+\frac{1}{\sqrt{2}\pi}\sum_{r=1}^\infty y_{jr}\sin(2\pi rt),
\end{equation}
where $x_{jr},~y_{jr}$ are $\N(0,1)$-random variables mutually independent for different values of $j=1,\cdots,q$ or $r\in\NN$, all independent of $W_1$. Then the double integral $I^\circ_{jk}=\int_0^1W_s^j\circ\td W_s^k$ has Fourier representation
\begin{equation}\label{double_int}
I^\circ_{jk}=\frac{1}{2}W_1^jW_1^k+\frac{1}{\sqrt{2}\pi}\lb W_1^jz_k-W_1^kz_k\rb+\frac{1}{2\pi}\lambda_{jk},
\end{equation}
where $\lambda_{jk}=\sum_{r\gq1}r^{-1}(x_{jr}y_{kr}-y_{jr}x_{kr})$ and $z_j=\sum_{r\gq1}r^{-1}x_{jr}$. One then needs to approximate each $\lambda_{jk}$ and $z_j$ by their partial sums $\lambda_{jk}=\sum_{r=1}^pr^{-1}(x_{jr}y_{kr}-y_{jr}k_{jr})$ and $z_j=\sum_{r=1}^pr^{-1}x_{jr}$. Denote $\wt\lambda_{jk}^{(p)}=\lambda_{jk}-\lambda_{jk}^{(p)},~\wt z_j^{(p)}=z_j-z_j^{(p)}$ and $U:=(\lambda,z),~U_p:=(\lambda^{(p)},z^{(p)}),~\wt U_p:=(\wt\lambda^{(p)},\wt z^{(p)})$.

Davie's result states that if there is a random variable $\bar{U}_p$, independent of $U_p$, having the same moments as $\wt U_p$ up to order $m-1$ and satisfying $\ex\exp(a\sqrt{p}|\bar{U}_p|)\lq b$ for some positive constants $a,b$ for all $p$, then $\wass_2(U,~U_p+\bar{U}_p)=O(p^{-m/2})$ for $p$ sufficiently large. The idea is to estimate the densities $g(\zeta)$ of $U$ and $h(\zeta)$ of $U_p+\bar{U}_p$. If $f_p$ is the density of $U_p$, then $g(\zeta)=\ex f_p(\zeta-\wt U_p)$ and $h(\zeta)=\ex f_p(\zeta-\bar{U}_p)$. By Taylor's theorem, for all $\zeta,w\in\RR^d$,
\begin{align}
f_p(\zeta-w)=&\sum_{|\beta|=0}^{m-1}\frac{(-1)^{|\beta|}}{\beta!}w^\beta\partial^\beta f_p(\zeta)\nonumber\\
&+\sum_{|\beta|=m}\frac{|\beta|(-1)^{|\beta|}}{\beta!}\int_0^1(1-\theta)^{|\beta|-1}w^\beta\partial^\beta f_p(\zeta-\theta w)\td\theta.\label{taylor_density}
\end{align}
Since up to the $(m-1)$-th moments of $\wt U_p$ and $\bar{U}$ match, when taking the difference $g(\zeta)-h(\zeta)$ the first summation vanishes, and hence $\forall\zeta\in\RR^d$,
\begin{equation}\label{different_g_h}
g(\zeta)-h(\zeta)=\sum_{|\beta|=m}\int_0^1C_{\beta,\theta}\ex\lb\wt U_p^\beta\partial^\beta f_p(\zeta-\theta\wt U_p)-\bar{U}^\beta\partial^\beta f_p(\zeta-\theta\bar{U}_p)\rb\td\theta,
\end{equation}
where $C_{\beta,\theta}=|\beta|(-1)^{|\beta|}(1-\theta)^{|\beta|-1}/\beta!$. If one can give a uniform bound for some higher derivatives of $f_p$ in terms of $p$, then using an interpolation argument one can show a reasonable decay for the $m$-th derivative of $f_p$, and finally one finds a coupling between $U$ and $U_p+\bar U_p$ by the inequality \eqref{crude_bound_1}.

The main advantage of the double integral $I^\circ_{jk}$ compared to the triple one is the fact that its Fourier representation only involves $\lambda$ and $z$, whose summands are independent. This ensures that $U$ has a smooth density (as the convolution of the density $f_p$ of $U_p$ and the law of $\wt U_p$), which significantly simplifies the analysis. More importantly, the characteristic function of $U_p$ can be explicitly calculated - see formula (32) in the proof of Lemma 11 in \cite{davie2014Patappstodifequusicou}. This provides some convenience for investigating the global and local behaviour of the density $f_p$ (Lemma 12, 13 and 14). In particular, Lemma 14 therein gives a lower bound for $f_p$, which is essential for achieving a coupling for $U$ of the optimal order $O(p^{-m/2})$ in the $\wass_2$ distance.

Without Lemma 14, one can still achieve a suboptimal $\wass_2$-rate $O(p^{-m/4})$ by directly showing a decay of the difference $|g(\zeta)-h(\zeta)|$ - this is the goal of the present paper, but the treatment of the densities is quite different from the double integral case.

The latter is much more straighforward to see. For $p$ sufficiently large, the vector $\tD^{2m}f_p$ of partial derivatives of order $2m$ is uniformly bounded everywhere due to part (1) of Lemma 11 in \cite{davie2014Patappstodifequusicou}. Also by Lemma 12 therein, one has $f_p(\zeta)\lq e^{-c_q|\zeta|}$ for $|\zeta|$ sufficiently large. Then one can apply Lemma 9 therein to get a rapid decay for $\tD^mf_p(\zeta)$. To see this, consider $|\zeta|>p$ sufficiently large and the ball $B(\zeta,1)$ that is disjoint with $B(0,p)$. Then $\sup_{y\in B(\zeta,1)}f_p(y)\lq e^{-c_q(|\zeta|-1)}$, and by applying Lemma 9 to the ball $B(\zeta,1)$ one sees the following bound for (the Euclidean norm of) the $m$-th derivatives:
\begin{equation}\label{interpolation}
|\tD^mf_p(\zeta)|\lq C_{q,m}\max\left\{\sup_{y\in B(\zeta,1)}\sqrt{f_p(y)}\sup_{y\in B(\zeta,1)}\sqrt{|\tD^{2m}f_p(y)|},~\sup_{y\in B(\zeta,1)}f_p(y)\right\}.
\end{equation}
This yields $|\tD^mf_p(\zeta)|\lq C_{q,m}e^{-c_q|\zeta|}$. Therefore from \eqref{different_g_h} and part (2) of Lemma 11 in \cite{davie2014Patappstodifequusicou} one has, by the Cauchy-Schwartz inequality, that for all $\zeta\in\RR^{q(q+1)/2}$,
\begin{align*}
|g(\zeta)-h(\zeta)|\lq&C_{d,m}\sum_{|\beta|=m}\lb\ex|\wt U_p^\beta\partial^\beta f_p(\zeta-\wt U_p)|+\ex|\bar{U}^\beta\partial^\beta f_p(\zeta-\bar{U}_p)|\rb\\
\lq&C_{d,m}p^{-m/2}\lb\sqrt{\ex|\tD^mf_p(\zeta-\wt U_p)|^2}+\sqrt{\ex|\tD^mf_p(\zeta-\bar{U}_p)|^2}\rb.
\end{align*}
Notice that, on the set $\{\omega:~|\wt U_p|\lq1\}$ one has $\|\tD^mf_p(\zeta-\wt U_p)\|^2\lq C_qe^{-c_q|\zeta|}$ by the rapid decay of $\tD^mf_p$; on the complement $\{\omega:~|\wt U_p|>1\}$, part (2) of Lemma 11 and Chebyshev's inequality imply that $\prb(|\wt U_p|>1)\lq C_Mp^{-M}$ for any $M>0$. The same argument works for the second term above involving $\bar{U}$, and so by the inequality \eqref{crude_bound_1} for the quadratic distance,
\begin{equation*}
\wass_2(U,~\wt U_p+\bar{U}_p)\lq C\lb\int_{\RR^{q(q+1)/2}}|\zeta|^2|g(\zeta)-h(\zeta)|\td\zeta\rb^{1/2}\lq C_{q,m}p^{-m/4}.
\end{equation*}

From this calculation one sees that the key step towards a good coupling result depends on how well the behaviour of $f_p$ is understood. Davie's result is a significant improvement to the existing rate of approximation - see the discussion following the proof of Theorem 15 therein. This is due to some careful estimates (Lemma 12, 13 and 14 in \cite{davie2014Patappstodifequusicou}) for the density $f_p$. For the triple integral $I^\circ_{jkl}$, however, showing similar estimates becomes much more complicated as the Fourier coefficients for $I^\circ_{jkl}$ have summands that are not independent of each other - see the definition of the random variable $\Delta_{jkl}$ below.

\paragraph{Notation.} Throughout this paper we will denote by $\phi$ the standard normal density of dimension $1$, by $B(x,r)$ the open ball of radius $r$ centred at $x$, and by $\Lambda^d$ the Lebesgue measure on $\RR^d$. The notation $C_0^\infty$ stands for the set of $C^\infty$-functions with compact support. Unless specified otherwise, the single bars $|\cdot|$ stand for the Euclidean norm, modulus of a complex number, or the cardinality of a set, and the double bars $\|\cdot\|$ stand for the operator norm, which in the context of matrices is equivalent to any other matrix norm. The letter $C$ will be used for a generic constant that may change value from line to line, with subscripts specifying its dependence on the parameters. The symbol $\lesssim_\alpha$ ($\gtrsim_\alpha$) means that the inequality $\lq$ ($\gq$) holds up to a multiplicative constant $C_\alpha$, and $\simeq_\alpha$ is used when both inequalities hold. For a function $f(x,y)$ of two variables, we also write $f(x;y)$ when, especially differentiating, $y$ is treated as a fixed parameter. For example, $\tD f(x;y)=\partial_xf(x,y)$.

\paragraph{Acknowledgement.} This work was completed under the patient guidance of my Ph.D. advisor, Prof. Alexander M. Davie, who suggested the problem and gave many crucial advices on the main steps of the arguments as well as technical details.

\section{The Fourier Representation}

For the simplicity of presentation let us consider the triple integral on the unit interval $[0,1]$. Following Section 5.8 in \cite{kloeden1995NumSolStoDifEqu}, from the Fourier expansion \eqref{bridge} the triple Stratonovich integral
\begin{equation*}
I^\circ_{jkl}=\int_0^1\int_0^tW_s^j\circ\td W^k_s\circ\td W^l_t,
\end{equation*}
for each $(j,k,l)\in\{1,\cdots,q\}^3$ has the following representation:
\begin{align*}
I^\circ_{jkl}=&\frac{1}{6}W_1^jW_1^kW_1^l-\frac{1}{2\sqrt{2}\pi}W_1^jW_1^k\lb z_l-\frac{1}{\pi}u_l\rb-\frac{1}{2\sqrt{2}\pi}W_1^kW_1^l\lb z_j-\frac{1}{\pi}u_j\rb\\
&-\frac{1}{\sqrt{2}\pi^2}W_1^jW_1^lu_k-\frac{1}{2\pi^2}z_j\lb W_1^kz_l-W_1^lz_k\rb+\frac{1}{2\pi}W_1^l\lb\frac{1}{2}\lambda_{jk}+\frac{1}{\pi}\nu_{kj}\rb\\
&+\frac{1}{2\pi}W_1^j\lb\frac{1}{2}\lambda_{kl}-\frac{1}{\pi}\nu_{kl}\rb+\frac{1}{4\pi^2}\lb W_1^j\mu_{kl}-W_1^k\mu_{jl}\rb-\frac{1}{2\sqrt{2}\pi^2}z_j\lambda_{kl}\\
&+\frac{1}{4\sqrt{2}\pi}\Delta_{jkl},
\end{align*}
where the coefficients $z,u,\lambda,\mu,\nu$ are defined as
\begin{align*}
z_j=&\sum_{r=1}^\infty\frac{1}{r}x_{jr},~u_j=\sum_{r=1}^\infty\frac{1}{r^2}y_{jr},\\
\lambda_{jk}=&\sum_{r=1}^\infty\frac{1}{r}\lb x_{jr}y_{kr}-y_{jr}x_{kr}\rb,~\mu_{jk}=\sum_{r=1}^\infty\frac{1}{r^2}\lb x_{jr}x_{kr}+y_{jr}y_{kr}\rb,\\
\nu_{jk}=&\sum_{\substack{r,s=1\\r\neq s}}^\infty\frac{1}{r^2-s^2}\lb\frac{r}{s}x_{jr}x_{ks}-y_{jr}y_{ks}\rb,
\end{align*}
with $x_{jr},y_{jr}$, again, being $\N(0,1)$-random variables independent for different indices $j=1,\cdots,q,~r\in\ZZ^+$ and all independent of $W_1^j$, and the last coefficient $\Delta$ is given by
\begin{align*}
\Delta_{jkl}=\sum_{r,s=1}^\infty\left\{-\frac{1}{r(r+s)}\left[(x_{jr}y_{ks}+y_{jr}x_{ks})x_{l,r+s}+(-x_{jr}x_{ks}+y_{jr}y_{ks})y_{l,r+s}\right]\right.&\\
+\frac{1}{rs}\left[(x_{jr}y_{ls}+y_{jr}x_{ls})x_{k,r+s}+(-x_{jr}x_{ls}+y_{jr}y_{ls})y_{k,r+s}\right]&\\
\left.+\frac{1}{s(r+s)}\left[(-x_{kr}y_{ls}+y_{kr}x_{ls})x_{j,r+s}+(x_{kr}x_{ls}+y_{kr}y_{ls})y_{j,r+s}\right]\right\}&
\end{align*}
For a positive integer $p$, write $z^{(p)}$ as the $p$-th partial sum of $z$ and $\wt z^{(p)}=z-z^{(p)}$. Similar notations are applied to $u,\lambda$ and $\mu$. Let $\nu^{(p)}$ be the partial sum of $\nu$ over $r,s\lq p,~r\neq s$ and $\wt\nu^{(p)}=\nu-\nu^{(p)}$, whilst $\Delta^{(p)}$ denotes the partial sum of $\Delta$ up to $r+s\lq p$ and $\wt\Delta^{(p)}=\Delta-\Delta^{(p)}$.

From the definition of $\nu_{jk}^{(p)}$ one observes that, by splitting each variable $\mu_{jk}^{(p)}$ into two parts:
\begin{equation*}
\mu_{jk}^{(1,p)}:=\sum_{r=1}^p\frac{1}{r^2}x_{jr}x_{kr},~\mu_{jk}^{(2,p)}:=\sum_{r=1}^p\frac{1}{r^2}y_{jr}y_{kr},
\end{equation*}
one need only generate $\nu_{jk}^{(p)}$ for $j<k$, since
\begin{equation*}
\nu_{jk}^{(p)}+\nu_{kj}^{(p)}=z_j^{(p)}z_k^{(p)}-\mu_{jk}^{(1,p)}.
\end{equation*}
Equivalent notations for the infinite sums are used by omitting the superscript $(p)$ and the identity still holds. Therefore one need only consider $\nu_{jk}$ for $j<k$.

Another observation is that one need not consider all possible choices of the $3$-tuple $(j,k,l)\in\{1,\cdots,q\}^3$ for $\Delta$; it suffices to focus on those terms with $(j,k,l)$ being a \textbf{Lyndon word} - a word that is strictly less than all of its proper right factors in the lexicographic order. This is due to the fact that all triple Stratonovich integrals $I^\circ_{jkl}$ can be expressed by the Lyndon words of length at most $3$ - see Corollary 3.3 in \cite{gaines1994AlgIteStoInt}.

For a word $w$ in a totally ordered set $A$, if it is the concatenation of two non-empty words $u,v\in A$, i.e. $w=uv$, then $v$ is called a proper right factor of $w$. For example, $(1,1,2)$ and $(1,3,2)$ are both Lyndon words but $(1,2,1)$ is not. By definition, a triple $(j,k,l)$ is a Lyndon word if and only if $j<k\wedge l$ or $j=k<l$. Denote by $\mathfrak{L}_{3,q}\subset\{1,\cdots,q\}^3$ the set of Lyndon words of length $3$, then according to \cite{gaines1994AlgIteStoInt} $|\mathfrak{L}_{3,q}|=(q^3-q)/3$.

As an analogue of the work by Davie \cite{davie2014Patappstodifequusicou} (Section 9), one seeks to approximate the variable $V=(z,u,\lambda,\mu,\nu,\Delta)$ by studying the distribution of the partial sums
\begin{equation*}
V_p=(z^{(p)},u^{(p)},\lambda^{(p)},\mu^{(p)},\nu^{(p)},\Delta^{(p)}),
\end{equation*}
and that of the remainder $\wt{V}_p:=(\wt z^{(p)},\wt u^{(p)},\wt\lambda^{(p)},\wt\mu^{(p)},\wt\nu^{(p)},\wt\Delta^{(p)})$. Note that for an $O(h^{3/2})$-approximation of the SDE \eqref{sde}, one also needs to simulate the double integrals \eqref{double_int} along with the triple ones. But they are determined by the variables $(z,\lambda)$, which are already included in $V$.

To develop an analogue of Davie's results in \cite{davie2014Patappstodifequusicou}, it is necessary to give some suitable moment estimates for the remainder $\wt V_p$. For simplicity denote the dimension of $V$ by
\begin{equation*}
d=2q^2+2q+(q^3-q)/3,
\end{equation*}
and denote by $v_p$ the $\RR^{2qp}$-vector consisting of $x_{jr},y_{ks}$ for $j,k=1,\cdots,q$ and $r,s=1,\cdots,p$.

For vectors $\omega:=(\alpha,\beta^{(1)},\beta^{(2)},\gamma,a,b,\rho)\in\RR^d$ and $v=(x_{jr},y_{jr})_{j,r}\in\RR^{2qp}$, define the cubic \textbf{phase function} $\Phi_p:\RR^{2qp}\times\RR^d\to\RR$ by
\begin{align}
\Phi_p(v;\omega)=&\sum_{j<k}\lb\alpha_{jk}\lambda_{jk}^{(p)}+\gamma_{jk}\nu_{jk}^{(p)}\rb+\sum_{j\lq k}\lb\beta_{jk}^{(1)}\mu_{jk}^{(1,p)}+\beta_{jk}^{(2)}\mu_{jk}^{(2,p)}\rb\nonumber\\
&+\sum_{j=1}^q\lb a_jz_j^{(p)}+b_ju_j^{(p)}\rb+ \sum_{(j,k,l)\in\mathfrak{L}_{3,q}}^q\rho_{jkl}\Delta_{jkl}^{(p)}.\label{phase}
\end{align}
Then by definition the characteristic function $\psi_p(\xi)$ of $V_p$ is given by
\begin{align*}
\psi_p(\xi)=&\int_{\RR^{2qp}}\exp\{i|\xi|\Phi_p(x,y;\omega_0)\}\prod_{j=1}^q\prod_{r=1}^p\phi(x_{jr})\phi(y_{jr})\td x \td y\\
=:&\int_{\RR^{2qp}}\exp\{i|\xi|\Phi_p(v;\omega_0)\}\phi_p(v)\td v,
\end{align*}
where $\phi$ is the density function of $\N(0,1)$ and $\omega_0=\xi/|\xi|\in\mathbb{S}^{d-1}$. Observe that the matrices $\lambda$ and $\mu$ are skew-symmetric and symmetric, respectively, so it would be convenient to extend the values of the coefficients $\alpha,~\beta:=(\beta^{(1)},~\beta^{(2)})$ to their lower-triangles by setting $\alpha_{kj}=-\alpha_{jk},~\beta_{kj}^{(i)}=\beta_{jk}^{(i)}$ for all $i=1,2,~j,k=1,\cdots,q$. Set $\gamma_{jk}=0$ for all $j\gq k$ and $\rho_{jkl}=0$ if $(j,k,l)$ is not a Lyndon word.

Throughout this article we will be frequently dealing with oscillatory integrals of the form $\psi_p(\xi)$, and we will conveniently call the function $\phi_p$ the \textbf{amplitude}. In order to give a good estimate for magnitude of $\psi_p(\xi)$ one resorts to the method of stationary phase, and for that one needs to study the derivatives of the phase function $\Phi_p$.

To find the gradient $\nabla\Phi_p(v;\omega)$, one can make use the extended definitions of $\alpha,\beta,\gamma$ and write down the partial derivatives. For each $j=1,\cdots,q$ and $r=1,\cdots,p$, differentiating w.r.t. $x_{jr}$ and $y_{jr}$ gives
\begin{align*}
\partial_{x_{jr}}&\Phi_p(v;\omega)=\frac{1}{r}\alpha_{jk}y_{kr}+\frac{1}{r^2}(1+\delta_{kj})\beta_{jk}^{(1)}x_{kr}+\sum_{\substack{s=1\\s\neq r}}^p\frac{1}{r^2-s^2}\lb\frac{r}{s}\gamma_{jk}-\frac{s}{r}\gamma_{kj}\rb x_{ks}+\frac{1}{r}a_j\\
&+\sum_{s=1}^{p-r}\left[\lb\frac{-\rho_{jkl}+\rho_{lkj}}{r(r+s)}-\frac{\rho_{kjl}}{s(r+s)}\rb y_{ks}x_{l,r+s}+\lb\frac{\rho_{jkl}+\rho_{lkj}}{r(r+s)}+\frac{\rho_{kjl}}{s(r+s)}\rb x_{ks}y_{l,r+s}\right.\\
&\qquad\quad\left.+\lb\frac{\rho_{jkl}+\rho_{lkj}}{rs}-\frac{\rho_{kjl}}{s(r+s)}\rb (y_{ls}x_{k,r+s}-x_{ls}y_{k,r+s})\right]\\
&+\sum_{s=1}^{r-1}\left[\lb-\frac{\rho_{jkl}}{rs}+\frac{\rho_{kjl}}{(r-s)s}\rb x_{k,r-s}y_{ls}+\lb\frac{\rho_{jkl}}{rs}+\frac{\rho_{kjl}}{(r-s)s}\rb y_{k,r-s}x_{ls}\right.\\
&\qquad\quad\left.-\frac{\rho_{lkj}}{(r-s)r}\lb x_{l,r-s}y_{ks}+y_{l,r-s}x_{ks}\rb\right],\num\label{1st_der_x}\\
\partial_{y_{jr}}&\Phi_p(v;\omega)=-\frac{1}{r}\alpha_{jk}x_{kr}+\frac{1}{r^2}(1+\delta_{kj})\beta_{jk}^{(2)}y_{kr}-\sum_{\substack{s=1\\s\neq r}}^p\frac{1}{r^2-s^2}(\gamma_{jk}-\gamma_{kj})y_{ks}+\frac{1}{r^2}b_j\\
&+\sum_{s=1}^{p-r}\left[\lb-\frac{\rho_{jkl}+\rho_{lkj}}{r(r+s)}-\frac{\rho_{kjl}}{s(r+s)}\rb x_{ks}x_{l,r+s}+\lb\frac{-\rho_{jkl}+\rho_{lkj}}{r(r+s)}-\frac{\rho_{kjl}}{s(r+s)}\rb y_{ks}y_{l,r+s}\right.\\
&\qquad\quad\left.+\lb\frac{\rho_{jkl}+\rho_{lkj}}{rs}+\frac{\rho_{kjl}}{s(r+s)}\rb \lb x_{ls}x_{k,r+s}+y_{ls}y_{k,r+s}\rb\right]\\
&+\sum_{s=1}^{r-1}\left[\lb\frac{\rho_{jkl}}{(r-s)r}-\frac{\rho_{kjl}}{(r-s)s}\rb x_{k,r-s}x_{ls}+\lb\frac{\rho_{jkl}}{rs}+\frac{\rho_{kjl}}{(r-s)s}\rb y_{k,r-s}y_{ls}\right.\\
&\qquad\quad\left.+\frac{\rho_{lkj}}{(r-s)r}\lb x_{l,r-s}x_{ks}-y_{l,r-s}y_{ks}\rb\right],\num\label{1st_der_y}
\end{align*}
where $\delta_{jk}$ is the Kr\"onecker delta, the summation signs over the repeated indices $k,l=1,\cdots,q$ are omitted, and all $x$ and $y$-terms with second subscripts outwith the interval $[1,p]$ are assumed to vanish. The Hessian matrix of $\Phi_p$ takes the form
\begin{equation}\label{hessian}
\tD^2\Phi_p(v;\omega)=\begin{pmatrix}
H_{xx}(1,1) & \cdots & H_{xx}(1,q) & H_{xy}(1,1) & \cdots & H_{xy}(1,q)\\
\vdots & \ddots & \vdots & \vdots & \ddots & \vdots\\
H_{xx}(q,1) & \cdots & H_{xx}(q,q) & H_{xy}(q,1) & \cdots & H_{xy}(q,q)\\
H_{yx}(1,1) & \cdots & H_{yx}(1,q) & H_{yy}(1,1) & \cdots & H_{yy}(1,q)\\
\vdots & \ddots & \vdots & \vdots & \ddots & \vdots\\
H_{yx}(q,1) & \cdots & H_{yx}(q,q) & H_{yy}(q,1) & \cdots & H_{yy}(q,q)
\end{pmatrix},
\end{equation}
where for each pair $(j,k)\in\{1,\cdots,q\}^2$ the blocks $H_{xx}(j,k),~H_{xy}(j,k),~H_{yy}(j,k)$ are $p\times p$ matrices, e.g.,
\begin{equation}\label{block}
H_{xx}(j,k)=\begin{pmatrix}
\partial^2_{x_{j1}x_{k1}} & \partial^2_{x_{j1}x_{k2}} & \cdots & \partial^2_{x_{j1}x_{kp}}\\
\partial^2_{x_{j2}x_{k1}} & \partial^2_{x_{j2}x_{k2}} & \cdots & \partial^2_{x_{j2}x_{kp}}\\
\vdots & \vdots & \ddots & \vdots\\
\partial^2_{x_{jp}x_{k1}} & \partial^2_{x_{jp}x_{k2}} & \cdots & \partial^2_{x_{jp}x_{kp}}
\end{pmatrix}\Phi_p(v,\omega),
\end{equation}
and the rest are similarly defined. From the gradient of $\Phi_p$ in $v$ one can compute the second derivative $\tD^2\Phi_p$ by finding the mixed derivatives for each pair $(j,k)$ and $(r,s)$. The $(r,s)$-th entries of the blocks $H_{xx}(j,k),~H_{yy}(j,k)$ and $H_{xy}(j,k)$ are given by
\begin{align*}
\partial^2_{x_{jr}x_{ks}}\Phi_p(v;\omega)=&\frac{1}{r^2}(1+\delta_{jk})\beta_{jk}^{(1)}\delta_{rs}+\frac{1}{r^2-s^2}\lb\frac{r}{s}\gamma_{jk}-\frac{s}{r}\gamma_{kj}\rb(1-\delta_{rs})\\
&+\lb\frac{\rho_{jkl}+\rho_{lkj}}{r(r+s)}+\frac{\rho_{kjl}+\rho_{ljk}}{s(r+s)}-\frac{\rho_{jlk}+\rho_{klj}}{rs}\rb y_{l,r+s}\\
&+\lb\frac{-\rho_{jlk}+\rho_{klj}}{rs}-\frac{\rho_{ljk}+\rho_{kjl}}{(s-r)s}+\frac{\rho_{jkl}+\rho_{lkj}}{r(s-r)}\rb y_{l,s-r}\\
&+\lb\frac{-\rho_{klj}+\rho_{jlk}}{rs}+\frac{\rho_{kjl}+\rho_{ljk}}{(r-s)s}-\frac{\rho_{jkl}+\rho_{lkj}}{(r-s)r}\rb y_{l,r-s},\num\label{2nd_der_xx}\\
\partial^2_{y_{jr}y_{ks}}\Phi_p(v;\omega)=&\frac{1}{r^2}(1+\delta_{jk})\beta_{jk}^{(2)}\delta_{rs}-\frac{1}{r^2-s^2}(\gamma_{jk}-\gamma_{kj})(1-\delta_{rs})\\
&+\lb\frac{-\rho_{jkl}+\rho_{lkj}}{r(r+s)}+\frac{-\rho_{kjl}+\rho_{ljk}}{s(r+s)}+\frac{\rho_{jlk}+\rho_{klj}}{rs}\rb y_{l,r+s}\\
&+\lb\frac{-\rho_{jlk}+\rho_{klj}}{rs}+\frac{-\rho_{ljk}+\rho_{kjl}}{(s-r)s}+\frac{\rho_{jkl}+\rho_{lkj}}{r(s-r)}\rb y_{l,s-r}\\
&+\lb\frac{\rho_{jlk}-\rho_{klj}}{rs}+\frac{\rho_{ljk}+\rho_{kjl}}{(r-s)s}+\frac{\rho_{jkl}-\rho_{lkj}}{(r-s)r}\rb y_{l,r-s},\num\label{2nd_der_yy}\\
\partial^2_{x_{jr}y_{ks}}\Phi_p(v;\omega)=&\frac{1}{r}\alpha_{jk}\delta_{rs}+\lb\frac{-\rho_{jkl}+\rho_{lkj}}{r(r+s)}-\frac{\rho_{kjl}+\rho_{ljk}}{s(r+s)}+\frac{\rho_{jlk}+\rho_{klj}}{rs}\rb x_{l,r+s}\\
&+\lb\frac{\rho_{jlk}+\rho_{klj}}{rs}+\frac{\rho_{ljk}+\rho_{kjl}}{(s-r)s}-\frac{\rho_{jkl}+\rho_{lkj}}{r(s-r)}\rb x_{l,s-r}\\
&+\lb-\frac{\rho_{jlk}+\rho_{klj}}{rs}+\frac{\rho_{ljk}+\rho_{kjl}}{(r-s)s}+\frac{-\rho_{lkj}+\rho_{jkl}}{(r-s)r}\rb x_{l,r-s},\num\label{2nd_der_xy}
\end{align*}
where, again, the summation sign over the repeated index $l=1,\cdots,q$ is omitted, and all $x$ and $y$-terms with second subscripts outwith the interval $[1,p]$ are assumed to vanish.

\section{The Joint Characteristic Function of the Partial Sums}

With the gradient and the Hessian matrix of the phase function $\Phi_p(v;\omega)$ in $v$ given above, one can apply the method of stationary phase to study the asymptotic behaviour of the oscillatory integral $\psi_p(\xi)$. A useful tool for this is provided in \cite{sogge1993FouIntClaAna} (Lemma 0.4.7), and the first estimate given in the following lemma is a more quantitative version of it.

Before stating the lemma let us introduce the norm
\begin{equation*}
|\varphi|_{K,\Omega}:=\max_{0\lq n\lq K}\sup_{x\in\Omega}|\tD^n\varphi(x)|
\end{equation*}
for any smooth function $\varphi$ on a bounded domain $\Omega\subset\RR^d$ and any natural number $K$.

\begin{lemma}\label{stationary_phase}
Let $\Psi,\varphi\in C^\infty(\RR^k)$ with $\supp\varphi=\Omega$ bounded. Then for all $\delta,R>0$ and $K\in\NN$,
\begin{equation*}
\lv\int_\Omega e^{iR\Psi(x)}\varphi(x)\td x\rv\lq C|\varphi|_{K,\Omega}\lb|\Psi|_{K,\Omega}^K\vee1\rb\delta^{-2K}R^{-K}+2|\varphi|_{0,\Omega}\Lambda^k(\Omega\setminus\Omega_\delta),
\end{equation*}
where $\Omega_\delta:=\{x\in\Omega:|\nabla\Psi(x)|>\delta\}$ and the constant $C$ depends on $k,K$ and $\Lambda^k(\Omega)$.
\end{lemma}
\begin{proof}
It suffices to bound the integral on $\Omega_\delta$. For any fixed $K>0$ write $M=|\Psi|_{K,\Omega}\vee1$ and split the set $\Omega_\delta$ into the level sets of the gradient of the phase function:
\begin{equation*}
\Omega_r:=\{x\in\Omega_\delta:~2^{-r}M<|\nabla\Psi(x)|\lq2^{-r+1}M\},
\end{equation*}
for $r=1,\cdots,r_0:=[\log_2(M/\delta)]$; there are at most $[\log_2(M/\delta)]+1$ non-empty $\Omega_r$'s. On each level set $\Omega_r$, choose $\eps_r=2^{-r}M/(M+1)$ and let $N_r=N_r(d,\eps_r)$ be the maximum number s.t. there are $x_1,\cdots,x_{N_r}\in\Omega_r$ so that the balls $B(x_j,\eps_r/2)$ are all disjoint. Then the balls $\{B(x_j,\eps_r)\}_j$ must cover $\Omega_r$: if there is $x_\ast\in\Omega_r$ s.t. $|x_\ast-x_j|>\eps_r$ for all $j$, then $B(x_\ast,\eps_r/2)$ is disjoint from all other balls $B(x_j,\eps_r)$ or those with half of the radius, which contradicts the maximality of $N_r$. Note that $\bigcup_{j=1}^{N_r}B(x_j,\frac{\eps_r}{2})\subset\Omega_r^\frac{\eps_r}{2}$, the $\frac{\eps_r}{2}$-neighbourhood of $\Omega_r$, therefore
\begin{equation*}
N_r\lq\frac{\Lambda^k\lb\Omega_r^\frac{\eps_r}{2}\rb}{\Lambda^k\lb B(x_j,\frac{\eps_r}{2})\rb}\lq C2^k\eps_r^{-k}\Lambda^k\lb\Omega^\frac{1}{4}\rb\lq C\eps_r^{-k},
\end{equation*}
where $C$ is a constant depending on $k$ and the Lebesgue measure of $\Omega$.

These balls altogether provide a finite open cover for the entire $\Omega_\delta$, on which there exist non-negative functions $\alpha_{j,r}\in C_0^\infty(B(x_j,\eps_r))$ that give a partition of unity: $\forall x\in\Omega_\delta$:
\begin{equation*}
\sum_{r=1}^{r_0}\sum_{j=1}^{N_r}\alpha_{j,r}(x)=1.
\end{equation*}
For each $r$, set further a smaller value for $\eps_r$ s.t. the balls covering $\Omega_s,~s\lq r$, do not intersect those of $\Omega_{r+2}$. Then one may choose, by Theorem 1.4.1 and 1.4.4 in \cite{hormander1990AnaLinParDifOpeI:DisTheFouAna}, such functions $\alpha_{j,r}$ that satisfy $|\alpha_{j,r}|_{K,B(x_j,\eps_r)}\lq C_{d,K}\eps_r^{-K}$ for each $j,r$ and any $K>0$. For each $j$ and $r$ define $\wt{\Psi}_{j,r}(y):=M^{-1}\eps_r^{-2}(\Psi(\eps_ry+x_j)-\Psi(x_j))$. Then for each $y\in B(0,1)$, the point $\eps_ry+x_j\in B(x_j,\eps_r)$, and by Taylor's theorem, there is some $x'\in B(x_j,\eps_r)$ s.t.
\begin{align*}
\lv\nabla\wt{\Psi}_{j,r}(y)\rv=&M^{-1}\eps_r^{-1}|\nabla\Psi(\eps_ry+x_j)|\\
\gq&M^{-1}\eps_r^{-1}|\nabla\Psi(x_j)|-\frac{1}{2}M^{-1}|\tD^2\Psi(x')|\\
\gq&\eps_r^{-1}2^{-r}-\frac{1}{2}>\frac{1}{2}.
\end{align*}
Since each $x_j\in\Omega_r$, by Taylor's theorem again, for all $y\in B(0,1)$ and some $x''\in B(x_j,\eps_r)$,
\begin{align*}
\lv\wt{\Psi}_{j,r}(y)\rv\lq&M^{-1}\eps_r^{-1}\lv\nabla\Psi(x_j)\rv+\frac{1}{2}M^{-1}|\tD^2\Psi(x'')|\\
\lq&\eps_r^{-1}2^{-r+1}+\frac{1}{2}\lq\frac{9}{2};
\end{align*}
the same argument gives the same upper bound for $|\nabla\wt\Psi_{j,r}(y)|$. For all $n\gq2$, one also has the Euclidean norm $|\tD^n\wt{\Psi}_{j,r}(y)|\lq M^{-1}\eps_r^{n-2}|\tD^n\Psi(x_j)|\lq1$. Therefore $\wt\Psi_{j,r}$ is in a (uniformly) bounded subset of $C^\infty(B(0,1))$.

Now that each function $\varphi_{j,r}:=\alpha_{j,r}\varphi$ is supported on the ball $B(x_j,\eps_r)$, the function $\psi_{j,r}(y):=\varphi_{j,r}(\eps_ry+x_j)$ is then supported on $B(0,1)$, satisfying $|\psi_{j,r}|_{K,B(0,1)}\lq C_{k,K}|\varphi|_{K,\Omega}$ for all $K,j,r$. Hence using the same arguments as in the proof of Lemma 0.4.7 in \cite{sogge1993FouIntClaAna} one sees:
\begin{align*}
\lv\int_{B(x_j,\eps_r)}e^{iR\Psi(x)}\varphi_{j,r}(x)\td x\rv=&\eps_r^k\lv\int_{B(0,1)}e^{iM\eps_r^2R\wt{\Psi}_{j,r}(y)}\varphi_{j,r}(\eps_ry+x_j)\td y\rv\\
\lq&C_{k,K}|\varphi|_{K,\Omega}M^{-K}\eps_r^{k-2K}R^{-K}.
\end{align*}
Finally, since $\supp\varphi=\Omega$, by the triangle inequality one deduces that
\begin{align*}
\lv\int_\Omega e^{iR\Psi(x)}\varphi(x)\td x\rv\lq&\sum_r\sum_{j=1}^{N_r}\lv\int_{B(x_j,\eps_r)}e^{iR\Psi(x)}\varphi_{j,r}(x)\td x\rv+|\varphi|_{0,\Omega}\Lambda^k\lb\bigcup_{r,j}B(x_j,\eps_r)\setminus\Omega_\delta\rb\\
\lq&C|\varphi|_{K,\Omega}M^{-K}\sum_rN_r\eps_r^{k-2K}R^{-K}+|\varphi|_{0,\Omega}\Lambda^k(\Omega\setminus\Omega_\delta)\\
\lq&C|\varphi|_{K,\Omega}M^{-K}R^{-K}\sum_r\eps_r^{-2K}+|\varphi|_{0,\Omega}\Lambda^k(\Omega\setminus\Omega_\delta)\\
\lq&C|\varphi|_{K,\Omega}M^K\delta^{-2K}R^{-K}+|\varphi|_{0,\Omega}\Lambda^k(\Omega\setminus\Omega_\delta),
\end{align*}
where $C$ is a constant depending on $k,K$ and $\Lambda^k(\Omega)$.
\end{proof}

This lemma is to be applied to $\Psi(v)=\Phi_p(v;\omega_0)$ and $\Omega_\delta=\{v\in\Omega,~|\nabla\Phi_p(v;\omega_0)|>\delta\}$ for some bounded domain $\Omega\subset\RR^{2qp}$ and any $\delta>0$; in this case the phase function $\Psi$ also depends on the parameter $\omega_0=\xi/|\xi|$. Instead of a unit vector consider $\omega\in\RR^d$ s.t. $|\omega|\gq c$ for some $c>0$: if the $v$-derivatives of $\Psi(v;\omega)$ have no singularity in $\omega$, then the result holds with $|\Psi|_{K,\Omega}$ replaced by $\sup_{|\omega|\gq c}|\Psi(\cdot;\omega)|_{K,\Omega}$. If the amplitude $\varphi$ also depends on $\omega$, then $|\varphi|_{K,\Omega}$ should be replaced by $\sup_{|\omega|\gq c}|\varphi(\cdot;\omega)|_{K,\Omega}$.

It then remains to estimate the Lebesgue measure of the exceptional set $\Omega\setminus\Omega_\delta$, which would also depend on $\omega$ if $\Psi=\Psi(v;\omega)$. The next three lemmas are devoted to this; the idea is to study the degeneracy of the Hessian matrix $\tD^2\Phi_p(v;\omega)$ described by \eqref{2nd_der_xx}, \eqref{2nd_der_xy} and \eqref{2nd_der_yy}. We start with the following general fact.

\begin{lemma}\label{factorisation}
Let $\Omega\subset\RR^k$ be open and bounded, $f:\Omega\to\RR^l$ be a $C^1$ function. For each $x$, let $\sigma_1(x)\gq\sigma_2(x)\gq\cdots\gq\sigma_{k\wedge l}(x)$ be the singular values of its derivative $\tD f(x)$. For any $n\in[1,k\wedge l]\cap\NN$ and $\eta>0$, define
\begin{equation*}
G_{n,\eta}(f):=\left\{x\in\Omega:\sigma_n(x)>\eta\right\}.
\end{equation*}
If $\tD f$ is Lipschitz continuous with Lipschitz constant $L$, then $\forall\delta>0$,
\begin{equation*}
\Lambda^d\lb G_{n,\eta}(f)\cap\{|f|\lq\delta\}\rb\lq CL^n\eta^{-2n}\delta^n,
\end{equation*}
where the constant $C$ depends on $k,l$ and $\Lambda^k(\Omega)$.
\end{lemma}
\begin{proof}
For fixed $n,\eta$ and any $z\in G_{n,\eta}(f)$, by definition the matrix $\tD f(z)$ has rank $n$. This implies that for each $z$ there are $n$-dimensional subspaces $E_z$ of $\RR^k$ and $F_z$ of $\RR^l$ s.t., with $g_z(\cdot):=\pi_{F_z}\circ f|_{E_z}(\cdot)$ and $\pi_\cdot$ being the orthogonal projection, the linear map $\tD g_z(z)$ is invertible. Denote by $E_z^\perp$ the orthogonal complement of $E_z$ for each $z$.

By the continuity of $\tD f$ the set $G_{n,\eta}(f)$ is open, and the inverse function theorem implies that $g_z$ is a diffeomorphism in some neighbourhood\footnote{The superscript $(n)$ signifies that it is a ball in $\RR^n$. Balls without superscripts lie in the whole space $\RR^k$.} $B^{(n)}(z,\eps)\subset E_z$. Moreover, in the proof of the inverse function theorem (see, e.g., Theorem 9.24 in \cite{rudin1976PriMatAna} or Theorem 1.1.7 in \cite{hormander1990AnaLinParDifOpeI:DisTheFouAna}), the ball $B^{(n)}(z,\eps)$ can be typically constructed with radius
\begin{equation*}
\eps\lq\frac{1}{2L|(\tD g_z(z))^{-1}|}\lq\frac{|\tD f(z)|}{2L}.
\end{equation*}
As $z\in G_{n,\eta}(f)$, one can choose e.g. $\eps=\eta/(4L)\wedge1$.

Since $G_{n,\eta}(f)$ is bounded, similar to the proof of Lemma \ref{stationary_phase} there are finitely many points $z_1,\cdots,z_{N_\eps}\in G_{n,\eta}(f)$ s.t. $G_{n,\eta}(f)\subset\bigcup_{j=1}^{N_\eps}B(z_j,\eps)$, with the number of balls satisfying
\begin{equation*}
N_\eps\lq\frac{\Lambda^k\lb G_{n,\eta}^{\eps/2}(f)\rb}{\Lambda^k\lb B(z_j,\eps/2)\rb}\lq C2^k\eps^{-k}\Lambda^k\lb\Omega^{1/2}\rb\lq C\eps^{-k},
\end{equation*}
for some constant $C$ depending on $k$ and $\Lambda^k(\Omega)$.

Write $\Gamma_j=B(z_j,\eps)\cap G_{n,\eta}(f)$ and $P_{j,\delta}=\pi_{E_{z_j}^\perp}(\Gamma_j\cap\{|f|\lq\delta\})$ for $\delta>0$. For each $(k-n)$-dimensional vector $(x_{n+1},\cdots,x_k)\in P_{j,\delta}$ let $S_{j,\delta}=S_{j,\delta}(x_{n+1},\cdots,x_k)$ be the corresponding `slice' of the set $\Gamma_j\cap\{|f|\lq\delta\}$ parallel to $E_{z_j}$. Then
\begin{equation*}
\Gamma_j\cap\{|f|\lq\delta\}=\bigcup_{(x_{n+1},\cdots,x_k)\in P_{j,\delta}}S_{j,\delta}.
\end{equation*}
Notice that all the singular values of $\tD g_z$ are greater than $\eta$ on $\Gamma_j$. Then by a change of coordinates and variables, one has that
\begin{align*}
\Lambda^k&\lb G_{n,\eta}(f)\cap\{|f|\lq\delta\}\rb\lq\sum_{j=1}^{N_\eps}\int_{\Gamma_j\cap\{|f|\lq\delta\}}\td x_1\cdots\td x_k\\
&=\sum_{j=1}^{N_\eps}\int_{P_{j,\delta}}\td x_{n+1}\cdots\td x_k\int_{S_{j,\delta}}\td x_1\cdots\td x_n\\
&=\sum_{j=1}^{N_\eps}\int_{P_{j,\delta}}\td x_{n+1}\cdots\td x_k\int_{g_{z_j}(\Gamma_j)\cap\{|y|\lq\delta\}}\lv\det(\tD g_{z_j})^{-1}(y)\rv\td y_1\cdots\td y_n\\
&\lq C\lb\min_{j}\inf_{x\in\Gamma_j}\lv\det\tD g_{z_j}(x)\rv\rb^{-1}\delta^n\sum_{j=1}^{N_\eps}\Lambda^{k-n}\lb B^{(k-n)}(z_j,\eps)\rb\\
&\lq C\eta^{-n}\delta^nN_\eps\eps^{k-n},
\end{align*}
where the constant $C$ depends on $k,l$ and $\Lambda^k(\Omega)$. Then the result follows from the bound for $N_\eps$ and the choice of $\eps$.
\end{proof}

Now write $G_{n,\eta}=G_{n,\eta}(\nabla\Phi_p(\cdot;\omega))$ as defined in Lemma \ref{factorisation} with $k=l=2qp$. One then needs to estimate the measure of the complement $\Omega\setminus G_{n,\eta}$ for suitable values of $\eta$ and $n\lq2qp$. From the expressions \eqref{2nd_der_xx}, \eqref{2nd_der_yy} and \eqref{2nd_der_xy} one sees that the behaviour of the second derivatives depends on the magnitude of the parameter $\rho$. Since the differentiation is done w.r.t. the variable $v$, the measure $\Lambda^{2qp}(\Omega\setminus G_{n,\eta})$ may depend on $\omega$, which for now we do not assume to be a unit vector.

The following result gives an estimate for the case where $\rho$ is not too small.

\begin{lemma}\label{rho_big}
Let $\Omega\subset\RR^{2qp}$ be bounded and $n\lq\sqrt{2p}/4$ be an integer. If $|\rho|>\eps$ for some fixed $\eps\in(0,|\omega|)$, then one has $\Lambda^{2qp}(\Omega\setminus G_{n,\eta})\lq C\eps^{1-2n}\eta^n$, where $C$ is a constant depending on $q,p,n$ and $\diam(\Omega)$.
\end{lemma}
\begin{proof}
It suffices to focus on a submatrix of $\tD^2\Phi_p(v;\omega)$ since $\wh{G}_{n,\eta}\subset G_{n,\eta}$ where $\wh{G}_{n,\eta}$ is similarly defined by the singular values of the submatrix. Since $|\rho|>\eps$, locate the (Lyndon) word $(j,k,l^\ast)$ that gives the maximum entry $|\rho_{jkl^\ast}|\gq\eps\sqrt{3/(q^3-q)}$. Then for the fixed pair $(j,k)$ we will focus on the submatrix $H_{xx}(j,k)$.

For a particular pair $(r,s)$, observe from \eqref{2nd_der_xx} that $\partial^2_{x_{jr}x_{ks}}\Phi_p(v;\omega)$ contains all the permutations of the word $(j,k,l)$ for each index $l$. Recall that all non-Lyndon entries of $\rho$ are defined to be $0$, and that if $(j,k,l)$ is a Lyndon word, we have either $j<k\wedge l$ or $j=k<l$. Thus for every Lyndon word $(j,k,l)$, out of the rest five permutations only one of $\rho_{jlk}$ and $\rho_{kjl}$ may not vanish, corresponding to those two cases respectively. If $j<k$ one has that
\begin{align*}
\partial^2_{x_{jr}x_{ks}}\Phi_p(v;\omega)=&\frac{1}{r^2}\beta_{jk}^{(1)}\delta_{rs}+\frac{1}{r^2-s^2}\lb\frac{r}{s}\gamma_{jk}-\frac{s}{r}\gamma_{kj}\rb(1-\delta_{rs})\\
&+\sum_{l>j}\lb\frac{\rho_{jkl}}{r(r+s)}-\frac{\rho_{jlk}}{rs}\rb y_{l,r+s}+\sum_{l>j}\lb\frac{\rho_{jkl}}{r|s-r|}-\frac{\rho_{jlk}}{rs}\rb\sign(s-r)y_{l,|s-r|}.
\end{align*}
Clearly, when $r\neq s$ the coefficients of $y_{l^\ast,r+s}$ and $y_{l^\ast,|s-r|}$ cannot vanish simultaneously. This is trivial if $j=k$, for one has instead
\begin{equation*}
\partial^2_{x_{jr}x_{js}}\Phi_p(v;\omega)=\frac{2}{r^2}\beta_{jk}^{(1)}\delta_{rs}+\sum_{l>j}\frac{\rho_{jjl}}{rs}(y_{l,r+s}+y_{l,|s-r|}).
\end{equation*}
This means that for fixed $r\neq s$ the entries of the submatrix $H_{xx}(j,k)$ involve different components $y_{l,r+s}$ and $y_{l,|s-r|}$ of the vector $y$. Let us combine these two cases and write
\begin{equation}\label{entry}
\partial^2_{x_{jr}x_{ks}}\Phi_p(v;\omega)=\kappa_{rs}+w_{rs}\cdot y
\end{equation}
with a constant term $\kappa_{rs}=\kappa_{rs}(\gamma_{jk},r,s)$ and coefficient $w_{rs}=w_{rs}(\rho_{jk\cdot},r,s)\in\RR^{qp}$.

For integers $n\lq m\lq\sqrt{p/2}-1$, one can choose $r_1,\cdots,r_m,s_1,\cdots,s_m\lq p$ s.t. $r_a\neq s_b$ and the integers $r_a+s_b,|r_c-s_d|$ are all different from one another for all choices of $a,b,c,d=1,\cdots,m$. For example, one may choose $r_a=a,~s_a=a(2m+1)$. In this case, the only choice of $(a,b,c,d)$ s.t. $r_a+s_b=r_c+s_d$, i.e. $(c-a)+(d-b)(2m+1)=0$, is that $a=c$ and $b=d$; the same for $r_a-s_b=r_c-s_d$. There is no choice of $(a,b,c,d)$ for the equation $(a+c)+(b-d)(2m+1)=0$ to hold so $r_a+s_b=s_c-r_d$ is never satisfied. Since we also require that all of them are no greater than $p$, it is necessary that $\max_{a,b}(r_a+s_b)=2m(m+1)\lq p$.

With this particular choice of $r_1,\cdots,r_m,s_1,\cdots,s_m$, one obtains an $m\times m$ submatrix $Q_m(y)=Q_m(y;\rho,\gamma)$ of $H_{xx}(j,k)$, of which each entry takes the form \eqref{entry}; write $\kappa_{ab}=\kappa_{r_as_b},~w_{ab}=w_{r_as_b},~a,b=1,\cdots,m$ for short. Then $|w_{ab}|\gq c_{q,p}\eps$ for each $(a,b)$, since for the particular case $l=l^\ast$ we have that $|\rho_{jjl^\ast}/(rs)|\gq c_q\eps/p^2$ and, by the maximality of $|\rho_{jkl^\ast}|$, that
\begin{equation*}
\lv\frac{\rho_{jkl^\ast}}{r(s-r)}-\frac{\rho_{jl^\ast k}}{rs}\rv\gq\frac{|\rho_{jkl^\ast}|}{r(s-r)}-\frac{|\rho_{jl^\ast k}|}{rs}\gq\frac{|\rho_{jkl^\ast}|}{s(s-r)}\gq c_q\frac{\eps}{p^2}.
\end{equation*}
Secondly, this particular choice of $\{r_a,s_b\}_{a,b}$ ensures that each entry of the submatrix $Q_m(y)$, translated by the constant $\kappa_{ab}$, is a linear combination of different components of $y$ that are \textit{all distinct} from those appearing in other entries; in other words, the $m^2$ vectors $\{w_{ab}\}_{a,b}$ are \textit{mutually orthogonal}. Denote the rows of $Q_m(y)$ by $q_1(y),\cdots,q_m(y)$, then each $q_a^\top(y)=\kappa_a+W_ay$ where $\kappa_a=(\kappa_{a1},\cdots,\kappa_{am})^\top$ and $W_a$ is the $m\times qp$ matrix consisting of the rows $w_{a1},\cdots,w_{am}$.

Now define for $a=1,\cdots,n$ the set
\begin{equation}\label{set_diff}
F_a:=\{(x,y)\in\Omega:~\dist(q_a(y),~\spa\{q_b(y):~b=1,\cdots,n,~b\neq a\})>\sqrt{n}\eta\},
\end{equation}
then $Q_m(y)$ has rank at least $n$ for $(x,y)\in\bigcap_{a=1}^nF_a$. Every point $(x,y)\in F_a$ satisfies
\begin{equation*}
\inf_{c_1,\cdots,c_n\in\RR}\lv \kappa_a-\sum_{b\neq a}c_b\kappa_b+\lb W_a-\sum_{b\neq a}c_bW_b\rb y\rv>\sqrt{n}\eta.
\end{equation*}
The mutual orthogonality of the vectors $\{w_{ab}\}_{a,b}$ implies the mutual orthogonality of the $m$ rows of the matrix $U_a:=W_a-\sum_{b\neq a}c_bW_b$, which thereby has a right inverse on an $m$-dimensional subspace $E_m$ of $\RR^{qp}$. Note also that each $|w_{ab}|\gq c_{q,p}\eps$ implies that $U_a$ restricted on $E_m$ has norm at least $c_{q,p}\eps$. Hence by the translation-invariance of the Lebesgue measure and the boundedness of $\Omega$, that for each $a$,
\begin{align*}
\Lambda^{2qp}(\Omega\setminus F_a)\lq&C|\det(U_a|_{E_m})|^{-1}(\sqrt{n}\eta)^{m-n+1}\\
\lq&C\|(U_a|_{E_m})\|^{-m}(\sqrt{n}\eta)^{m-n+1}\lq C\eps^{-m}(\sqrt{n}\eta)^{m-n+1},
\end{align*}
where the constant $C=C(q,p,m,\diam(\Omega))$ grows at most exponentially in $m$.

For each point $(x,y)\in\bigcap_{a=1}^nF_a$ and any unit vector $e=(e_1,\cdots,e_n)$, consider the linear combination $e\cdot(q_1(y),\cdots,q_n(y))$ of the $n$ rows. Choose $a$ s.t. $|e_a|=\max\{|e_1|,\cdots,|e_n|\}\gq1/\sqrt{n}$, then
\begin{equation*}
|e_1q_1(y)+\cdots+e_nq_n(y)|=|e_a|\lv q_a(y)+\sum_{b\neq a}e_a^{-1}e_bq_b(y)\rv\gq\eta.
\end{equation*}
Thus, the $n\times m$ submatrix $\wh Q_n(y):=(q_1(y)^\top,\cdots,q_n(y)^\top)^\top$ has a right inverse $R_n(y)$ on an $n$-dimensional subspace $E_n$ of $\RR^m$, and
\begin{equation*}
\lvv R_n(y)\rvv=\sup_{|e|=1}\lv R_n(y)e^\top\rv\lq\lb\inf_{|e|=1}\lv e\wh Q_n(y)\rv\rb^{-1}\lq \eta^{-1}.
\end{equation*}
It then follows from the singular-value decomposition that the singular values of the matrix $\wh Q_n(y)$ are all bounded from below by $\|R_n(y)\|^{-1}\gq\eta$, which in turn gives an estimate for the measure of the exceptional set:
\begin{equation*}
\Lambda^{2qp}(\Omega\setminus G_{n,\eta})\lq\Lambda^{2qp}\lb\Omega\setminus\wh{G}_{n,\eta}\rb\lq\Lambda^{2qp}\lb\bigcup_{a=1}^n(\Omega\setminus F_a)\rb\lq Cn\eps^{-m}(\sqrt{n}\eta)^{m-n+1},
\end{equation*}
and the result follows by taking $m=2n-1$.
\end{proof}

The result of Lemma \ref{rho_big} is meaningful for small values of $\eps$ and $\eta$. It remains to show that the measure $\Lambda^{2qp}(\Omega\setminus G_{n,\eta})$ is also small when $\rho$ is small.

\begin{lemma}\label{rho_small}
Let $\Omega\subset\RR^{2qp}$ be bounded and $n\lq p$ be an even integer s.t. $n+1$ is prime. Then, depending on $q,p,n$ and $\diam(\Omega)$, one can choose $\delta,\eta,\eps\lesssim_{q,p,n}|\omega|$ sufficiently small s.t. for $|\rho|\lq\eps$, either $\Omega_\delta=\Omega$ or $G_{n,\eta}=\Omega$.
\end{lemma}
\begin{proof}
For $\eps\in(0,|\omega|/\sqrt{2})$ define $\bar\eps=\sqrt{|\omega|^2-\eps^2}\in(|\omega|/\sqrt{2},|\omega|)$, and assume $\diam(\Omega)=1$ w.l.o.g., otherwise replace $\eps$ with $\eps/(1\vee\diam(\Omega))$. First of all that $|\rho|\lq\eps$ implies that the vector $(\alpha,\beta,\gamma,a,b)$ has modulus no less than $\bar\eps$. The proof is divided into several cases depending on which components of this vector are dominant in modulus or norm.

Let us first consider the case where the coefficients $(a,b)$ are `dominant' in the sense that $|(a,b)|>\bar\eps\sqrt{1-\theta^2}\gq|\omega|/2$ for some $\theta\in(0,1/\sqrt{2})$ to be chosen later. In this case $|(\alpha,\beta,\gamma)|\lq\bar\eps\theta$. From the expression \eqref{1st_der_x} for the first derivatives one has the following bound:
\begin{equation}\label{1st_der_bound}
\lv\partial_{x_{jr}}\Phi_p(v;\omega)\rv^2\gq\frac{1}{r^2}a_j^2-\frac{2}{r}|a_j||Q_{x_{jr}}(v;\rho)|-2\lb\frac{1}{r}|a_j|+|Q_{x_{jr}}(v;\rho)|\rb|L_{x_{jr}}(v;\alpha,\beta,\gamma)|,
\end{equation}
where $L_{x_{jr}}(v;\alpha,\beta,\gamma)$ and $Q_{x_{jr}}(v;\rho)$ denote the linear and quadratic parts for $v$ in \eqref{1st_der_x}. Since $x$ and $y$ are bounded, one has that
\begin{equation}\label{1st_der_quadratic}
\sup_{v\in\Omega}|Q_{x_{jr}}(v;\rho)|\lq C_q|\rho|\frac{1}{r}\lb\sum_{s=1}^{p-r}\frac{1}{s}+\sum_{s=1}^{r-1}\frac{1}{s}\rb\lq C_q\frac{\eps}{r}\log p,
\end{equation}
and that, omitting the summation in $k$,
\begin{align}
\sup_{v\in\Omega}|L_{x_{jr}}(v;\alpha,\beta,\gamma)|\lq&\sup_{v\in\Omega}\lb\frac{1}{r}|\alpha_{jk}||y_{kr}|+\frac{2}{r^2}|\beta_{jk}^{(1)}||x_{kr}|+\frac{1}{r}(|\gamma_{jk}|+|\gamma_{kj}|)\sum_{s\neq r}\frac{1}{s}|x_{ks}|\rb\nonumber\\
\lq&C_q\frac{\log p}{r}\bar\eps\theta.\label{1st_der_linear}
\end{align}
Hence one derives that
\begin{equation*}
\inf_{v\in\Omega}|\partial_{x_{jr}}\Phi_p(v;\omega)|^2\gq\frac{1}{r^2}a_j^2-C_q\frac{\eps\log p}{r^2}|a_j|-C_q\bar\eps\theta\frac{\log p}{r}\lb\frac{1}{r}|a_j|+\frac{\eps}{r}\log p\rb,
\end{equation*}
and a similar inequality for $|\partial_{y_{jr}}\Phi_p|^2$ with $a_j/r$ replaced with $b_j/r^2$ as per \eqref{1st_der_y}. Thus, summing up $j$ and $r$ one has that
\begin{align}
\inf_{v\in\Omega}|\nabla\Phi_p(v;\omega)|^2\gq&C|(a,b)|^2-C_q|(a,b)|\eps\log p-C_q\bar\eps\theta(\log p)^2(|(a,b)|+\eps\log p)\nonumber\\
\gq&C|\omega|^2-C_q|\omega|\eps\log p-C_q|\omega|\theta(\log p)^2(|\omega|+\eps\log p),\label{grad_bound}
\end{align}
which has a fixed lower bound for $\eps<|\omega|(\log p)^{-1}$ and $\theta<(\log p)^{-2}$ sufficiently small. Then for small values of $\delta<C|\omega|$ we have $\Omega=\Omega_\delta$.

Now suppose that $|(a,b)|\lq\bar\eps\sqrt{1-\theta^2}$, then $|(\alpha,\beta,\gamma)|\gq\bar\eps\theta$. The latter corresponds to the constant terms in the second derivative $\tD^2\Phi_p(v;\omega)$. Write
\begin{equation}\label{2nd_der_mat}
\tD^2\Phi_p(v;\omega)=A_p+L_p(v;\rho)
\end{equation}
according to the expressions \eqref{2nd_der_xx}, \eqref{2nd_der_yy} and \eqref{2nd_der_xy}, where $A_p=A_p(\alpha,\beta,\gamma)$ and $L_p(v;\rho)=\{(L_{x_{jr}x_{ks}},L_{y_{jr}y_{ks}},L_{x_{jr}y_{ks}})(v;\rho)\}_{j,k,r,s}$ are the constant and linear parts in $v$, respectively. Then for each $(j,k)$ and $(r,s)$,
\begin{equation*}
\sup_{v\in\Omega}|L_{x_{jr}x_{ks}}(v;\rho)|\lq C_q|\rho|\lb\frac{1}{rs}+\frac{\delta_{rs}}{r|r-s|}+\frac{\delta_{rs}}{s|r-s|}\rb\lq C_q\frac{\eps}{r\wedge s},
\end{equation*}
and the same bound holds for $L_{y_{jr}y_{ks}}$ and $L_{x_{jr}y_{ks}}$. Let $H_n(v;\omega)=A_n+L_n(v;\rho)$ be an $n\times n$ submatrix of $\tD^2\Phi_p(v;\omega)$ with `constant' part $A_n=A_n(\alpha,\beta,\gamma)$ and linear part $L_n(v;\rho)$. Then the estimate above implies that
\begin{equation}\label{Ln_bound}
\sup_{v\in\Omega}\|L_n(v;\rho)\|\lq C_{q,p}\eps.
\end{equation}
Thus if $A_n$ is invertible, one can choose $\eps\lesssim_{q,p}\|A_n\|$ sufficiently small s.t. $H_n$ is also invertible for all $v\in\Omega$. Furthermore, choose $\eps<\|A_n^{-1}\|^{-1}$ small enough so that
\begin{equation}\label{Hn_lower_bound}
\|H_n^{-1}\|\lq\|A_n^{-1}\|\|(I+A_n^{-1}L_n)^{-1}\|\lq\|A_n^{-1}\|/(1-\|A_n^{-1}\|\|L_n\|)\lq2\eps^{-1},
\end{equation}
for all $v\in\Omega$. This follows from the fact that $\|(I+B)^{-1}\|=\|\sum_{k\gq0}(-B)^k\|\lq\sum_{k\gq0}\|B\|^k=1/(1-\|B\|)$ for any square matrix $B$ s.t. $\|B\|<1$. Henceforth by the singular-value decomposition the estimate \eqref{Hn_lower_bound} will imply that $H_n(v;\omega)$ - as an $n\times n$ submatrix of $\tD^2\Phi_p(v;\omega)$ - has singular values no less than $\eps/2$ for all $v\in\Omega$, in other words, $\Omega\setminus G_{n,\tau/2}=\varnothing$. In particular, for any $\eta\lq\eps/2$ we have $\Omega\setminus G_{n,\eta}=\varnothing$, too. Henceforth, one looks for an invertible $n\times n$ submatrix $A_n$ of $A_p$ with an appropriate bound for $\|A_n^{-1}\|$, and the result will follow by choosing sufficiently small values of $\eps$ and $\eta$.

Write $D_n=\diag(1,1/2,\cdots,1/n),~n\lq p$ for simplicity. If the component $\alpha$ is `dominant' amongst $\alpha,\beta,\gamma$ in the sense that, for example, $\|\alpha\|>\bar\eps\theta/\sqrt{3}>|\omega|\theta/\sqrt{6}$, choose the largest entry $|\alpha_{jk}|\gq c_q|\omega|\theta$. Then by \eqref{2nd_der_xy} the constant part of the $n$-th principle submatrix of the block $H_{xy}(j,k)$ is $A_n=\alpha_{jk}D_n$, and $\|A_n^{-1}\|\lq|\alpha_{jk}|^{-1}n$. Thus the result holds for $\eps\lesssim_{q,p}|\omega|\theta/n$.

On the other hand we need to consider the case where $|(\beta,\gamma)|\gq\bar\eps\theta\sqrt{2/3}$. If the largest entry of $(\beta,\gamma)$ is located on the diagonal, i.e. $|\beta_{jj}^{(i)}|\gq c_q\bar\eps\theta$ (recall that $\gamma_{jj}=0$) for $i=1$ or $2$ and some $j$, then the constant part of the $n$-th principle submatrix of the block $H_{xx}(j,j)$ or the block $H_{yy}(j,j)$ is $A_n^{(i)}=2\beta_{jj}^{(i)}D_n^2$ by \eqref{2nd_der_xx} and \eqref{2nd_der_yy}. Hence we have that $\|A_n^{-1}\|\lq|2\beta_{jj}^{(i)}|^{-1}n^2$ and we need $\eps\lesssim_{q,p}|\omega|\theta/n^2$.

The situation is trickier when the largest entry is found off the diagonal, i.e. for some pair $(j,k)$ (assuming $j<k$ w.l.o.g.). Consider the constant part $A_n^{(2)}$ of the $n$-th principle submatrix of the block $H_{yy}(j,k)$. By \eqref{2nd_der_yy} it takes the form
\begin{equation*}
A_n^{(2)}=\begin{pmatrix}
\beta_{jk}^{(2)} & \frac{1}{3}\gamma_{jk} & \frac{1}{8}\gamma_{jk} & \cdots & \frac{1}{n^2-1}\gamma_{jk} \\
-\frac{1}{3}\gamma_{jk} & \frac{1}{4}\beta_{jk}^{(2)} & \frac{1}{5}\gamma_{jk} & \cdots & \frac{1}{n^2-4}\gamma_{jk} \\
-\frac{1}{8}\gamma_{jk} & -\frac{1}{5}\gamma_{jk} & \frac{1}{9}\beta_{jk}^{(2)} & \cdots & \frac{1}{n^2-9}\gamma_{jk} \\
\vdots & \vdots & \vdots & \ddots & \vdots \\
-\frac{1}{n^2-1}\gamma_{jk} & -\frac{1}{n^2-4}\gamma_{jk} & -\frac{1}{n^2-9}\gamma_{jk} & \cdots & \frac{1}{n^2}\beta_{jk}^{(2)}
\end{pmatrix}=\beta_{jk}^{(2)}D_n^2+\gamma_{jk}S_n,
\end{equation*}
where $S_n$ is the skew-symmetric matrix with $(r,s)$-th entry $(s^2-r^2)^{-1},~r\neq s$ and $0$ on the diagonal. If $|\beta_{jk}^{(2)}|\gq c_q\bar\eps\theta$, then the matrix $A_n^{(2)}$ has full rank. To see this, notice that the matrix $\bar{S}_n:=D_n^{-1}S_nD_n^{-1}$ is also skew-symmetric and has purely imaginary eigenvalues only. Then all the eigenvalues of the scaled matrix $\bar{A}_n^{(2)}:=I+\gamma_{jk}\bar{S}_n/\beta_{jk}^{(2)}$ have real parts $1$, which serves as a lower bound for the operator norm of $\bar A_n^{(2)}$ as it is in fact a normal matrix, and so $\|(\bar{A}_n^{(2)})^{-1}\|\lq1$. Therefore $\|(A_n^{(2)})^{-1}\|=\|(\beta_{jk}^{(2)}D_n\bar{A}_n^{(2)}D_n)^{-1}\|\lq|\beta_{jk}^{(2)}|^{-1}n^2$ and again we need $\eps\lesssim_{q,p}|\omega|\theta/n^2$.

The same applies to the case where $|\beta_{jk}^{(1)}|\gq c_q\bar\eps\theta$: instead of $A_n^{(2)}$ consider the constant part $A_n^{(1)}$ of the $n$-th principle submatrix of the block $H_{xx}(j,k)$, which by \eqref{2nd_der_xx} takes the form $A_n^{(1)}=\beta_{jk}^{(1)}D_n^2+\gamma_{jk}S'_n$ where $S'_n$ is the matrix with $(r,s)$-th entry $(r^2-s^2)^{-1}r/s$.  Then it suffices to observe that $S'_n=-D_n^{-1}S_nD_n$ and $A_n^{(1)}=\beta_{jk}^{(1)}(I-\gamma_{jk}\bar{S}_n/\beta_{jk}^{(1)})D_n^2$.

Finally, if $|\gamma_{jk}|\gq c_q\bar\eps\theta$ is the largest entry of $(\beta,\gamma)$, we return to the matrix $A_n^{(2)}$. Since $S_n$ is skew-symmetric, $\det S_n=0$ for all odd $n$. If $n$ is even, by definition the determinant of $S_n$ is given by the expansion
\begin{equation*}
\det S_n=\sum_{\sigma\in\Pi_n^\ast}\sign(\sigma)\frac{1}{1-\sigma_1^2}\frac{1}{4-\sigma_2^2}\cdots\frac{1}{n^2-\sigma_n^2},
\end{equation*}
where $\Pi_n^\ast$ is the set of permutations of $(1,\cdots,n)$ with no fixed points. Notice that this summation includes the product of all the entries along the reflected diagonal $r+s=n+1$, each of which has denominator divisible by $n+1$. Clearly, out of all the permutations this product is the only term in the above expansion whose denominator is divisible by $(n+1)^n$ if $n+1$ is prime. Then it follows from the fundamental theorem of arithmetic that $\text{\thorn}_n:=\det S_n\neq0$. It is rather difficult to compute the the value $\text{\thorn}_n$ explicitly; computer results for large values of $n$ up to $400$ shows that it decays roughly exponentially. Notice that $A_n^{(2)} =D_n(I+\beta_{jk}^{(2)}\gamma_{jk}^{-1}\bar{S}_n^{-1})D_n^{-1}\gamma_{jk}S_n$ and that the matrix $\bar{S}_n^{-1}$ is still skew-symmetric, the same argument used in the previous cases still applies. Therefore $\|(A_n^{(2)})^{-1}\|\lq|\gamma_{jk}|^{-1}\|S_n^{-1}\|n$, and we need $\eps\lesssim_{q,p}|\omega|\theta\text{\thorn}_n^{1/n}/n$.

Combining all the criteria above, for an even integer $n$ s.t. $n+1$ is prime one can choose $\eps\lesssim_{q,p}|\omega|\text{\thorn}_n^{1/n}n^{-2}$ s.t. the result holds true for $\delta\lesssim_q|\omega|/4$ and $\eta<\eps/2$ sufficiently small.
\end{proof}

These lemmas altogether give an estimate for oscillatory integrals of the type
\begin{equation*}
T(R,\omega)=\int_\Omega e^{iR\Phi_p(v;\omega)}\varphi(v)\td v,
\end{equation*}
for a bounded domain $\Omega$ and a smooth amplitude $\varphi$ supported on $\Omega$. In order to study the global behaviour of it, in particular, the characteristic function $\psi_p(\xi)$ of $V_p$, some cut-off arguments will be needed to derive a similar estimate as in Lemma \ref{stationary_phase} on the whole space $\RR^{2qp}$.

But as the reader will realise later, to find a desired coupling for $V_p$ it is necessary to estimate oscillatory integrals with amplitudes other than just $\phi_p$. For a Schwartz function $\varphi\in\mathscr{S}(\RR^q)$ and $k,l\in\NN$, introduce the norm
\begin{equation*}
\|\varphi\|_{j,k}=\max_{|\theta|\lq j,|\tau|\lq k}\sup_{x\in\RR^q}|x^\theta\partial^\tau\varphi(x)|,
\end{equation*}
where $\theta,\tau\in\NN^q$ are multi-indices. Then for $\varphi\in C_0^\infty(\Omega)$ it holds that $|\varphi|_{k,\Omega}\simeq_q\|\varphi\|_{0,k}$.

\begin{theorem}\label{global_estimate}
For any $K>0$, let $p_0>8K^2$ be an even integer s.t. $[\sqrt{2p_0}/4]+1$ is a prime number and let $\varkappa=(2q+1/2)p_0+K+1$. For any $p\gq p_0,~\xi\in\RR^d,~\omega_0:=\xi/|\xi|$ and a Schwartz function $\varphi\in\mathscr{S}(\RR^{2qp})$, define
\begin{equation*}
I_p(\xi)=\int_{\RR^{2qp}}e^{i|\xi|\Phi_p(v;\omega_0)}\varphi(v)\td v,
\end{equation*}
and separate the phase function $\Phi_p$ in terms of distinct $v_0$-monomials:
\begin{equation}\label{separate_phase}
\Phi_p(v;\omega_0)=\sum_{|\beta|\lq3}v_0^{\beta}P_\beta(v'),
\end{equation}
where for each multi-index $\beta$ the polynomial $P_\beta$ has degree $3-|\beta|$. If $\varphi$ can be factorised as the product of two further Schwartz functions $\varphi_0\in\mathscr{S}(\RR^{2qp_0})$ and $\varphi_1\in\mathscr{S}(\RR^{2qp'}),~p':=p-p_0$, then $I_p\in C^\infty(\RR^d)$, and for any $k\in\NN$ and $|\xi|$ sufficiently large it holds that
\begin{equation}\label{DI_estimate}
|\tD^kI_p(\xi)|\lq C_{q,p_0,k,K}|\xi|^{-\frac{K}{16}}\|\varphi_0\|_{\varkappa+3k,K}\int_{\RR^{2qp'}}\lb1+\sum_{|\beta|\lq3}|P_\beta(v')|^{\sqrt{2p_0}+k-2}\rb\varphi_1(v')\td v',
\end{equation}
provided that the last integral is finite.
\end{theorem}
As we shall see later, this result is only useful when the integral in \eqref{DI_estimate} is indepdent of $p$.
\begin{proof}
Let us prove the rapid decay of $|I_p(\xi)|$ first. Instead of $I_p(\xi)$ let us consider for now the oscillatory integral
\begin{equation*}
I_{p_0}(\xi)=\int_{\RR^{2qp_0}}e^{i|\xi|\Phi_{p_0}(v_0;\omega_0)}\varphi_0(v_0)\td v_0
\end{equation*}
for a fixed positive integer $p_0$, and write $\omega_0=(a,b,\alpha,\beta,\gamma,\rho)\in\mathbb{S}^{d-1}$. First choose a non-negative, smooth cut-off function $\zeta_0\in C_0^\infty(B(0,2))$ s.t. $\zeta_0\equiv1$ on $B(0,1)$ and all its derivatives are bounded on $B(0,2)\setminus B(0,1)$. Divide the rest of $\RR^{2qp_0}$ by the annuli
\begin{equation*}
A_r:=\{v_0\in\RR^{2qp_0}:~2^{r-1}\lq|v_0|<2^r\},~r\in\NN,
\end{equation*}
and define the fattened annuli $A'_r:=\{2^{r-2}\lq|v_0|<2^{r+1}\}$. Choose another non-negative, smooth cut-off $\zeta_1\in C_0^\infty(A'_1)$ taking value $1$ on $A_1$ and bounded derivatives on $A'_1\setminus A_1$, and define $\zeta_r(v_0):=\zeta_1(2^{-r+1}v_0),~\forall r\gq2$. Then for each $r\gq1$, the smooth function $\zeta_r$ is supported on the fattened annulus $A'_r$ with value $1$ on $A_r$ and bounded derivatives on $A'_r\setminus A_r$; the sum $\sigma(v_0):=\sum_{r=0}^\infty\zeta_r(v_0)$ is then supported on the whole of $\RR^{2qp_0}$.

If one further sets $\wt{\zeta}_r(v_0):=\zeta_r(v_0)/\sigma(v_0)$, then each $\wt{\zeta}_r$ has the same properties as those of $\zeta_r$, and $\sum_{r=0}^\infty\wt{\zeta}_r\equiv1$ trivially. Then one can write
\begin{align*}
I_{p_0}(\xi)&=\int_{B(0,2)}e^{i|\xi|\Phi_{p_0}(v_0;\omega_0)}\varphi_0(v_0)\wt{\zeta}_0(v_0)\td v_0+\sum_{r=1}^\infty\int_{A'_r}e^{i|\xi|\Phi_{p_0}(v_0;\omega_0)}\varphi_0(v_0)\wt{\zeta}_r(v_0)\td v_0\\
&=:T_0(\xi)+\sum_{r=1}^\infty T_r(\xi),
\end{align*}
where the first integral can be readily estimated by the lemmas above. Since the function $\Phi_{p_0}(v_0;\omega_0)$ is a cubic polynomial and the vector $\omega_0=(a,b,\alpha,\beta,\gamma,\rho)$ is normalised, all the derivatives of $\Phi_{p_0}$ are uniformly bounded on $B(0,2)$; so do all the derivatives of $\wt\zeta_0$ by its construction.

Thus, applying Lemma \ref{stationary_phase} we have, $\forall K,\delta_0>0,~\xi\in\RR^d$,
\begin{equation*}
|T_0(\xi)|\lq C_{q,p_0,K}|\varphi_0|_{K,B(0,2)}\delta_0^{-2K}|\xi|^{-K}+2|\varphi_0|_{0,B(0,2)}\Lambda^{2qp_0}(\Gamma_{\delta_0}),
\end{equation*}
where $\Gamma_{\delta_0}=\{v_0\in B(0,2):|\nabla\Phi_{p_0}(v_0,\omega_0)|\lq\delta_0\}$. The set $\Gamma_{\delta_0}$ can be further split by the set $G_{n,\eta_0}=G_{n,\eta_0}(\nabla\Phi_{p_0})$ as defined in Lemma \ref{factorisation} and its complement for some $\eta_0>0$ and some integer $n$. Note that the Lipschitz constant of $\tD^2\Phi_{p_0}$ is at most $|\rho|\lq1$. Then by Lemma \ref{factorisation}, \ref{rho_big} and \ref{rho_small} one sees that for any $\delta_0,\eta_0,\eps_0>0$ sufficiently small and any even integer $n\lq\sqrt{2p_0}/4$ s.t. $n+1$ is prime, one has that
\begin{align}
\Lambda^{2qp_0}(\Gamma_{\delta_0})&\lq\Lambda^{2qp_0}(\Gamma_{\delta_0}\cap G_{n,\eta_0})+\Lambda^{2qp_0}(\Gamma_{\delta_0}\setminus G_{n,\eta_0})\nonumber\\
&\lesssim_{q,p_0,n}\eta_0^{-2n}\delta_0^n+\eps_0^{1-2n}\eta_0^n.\label{est_excep}
\end{align}
Thus, choosing $\eta_0=\delta_0^{1/4},~\delta_0=|\xi|^{-1/4}$ and $\eps_0\ll_{q,p_0}1$ one has that
\begin{align*}
|T_0(\xi)|\lq&C_{q,p_0,K}|\varphi_0|_{K,B(0,2)}|\xi|^{-\frac{1}{2}K}+C_{q,p_0,n}|\varphi_0|_{0,B(0,2)}\lb|\xi|^{-\frac{1}{8}n}+|\xi|^{-\frac{1}{16}n}\rb\\
\lq&C_{q,p_0,n,K}\|\varphi_0\|_{0,K}|\xi|^{-\frac{1}{16}K}
\end{align*}
for $|\xi|$ sufficiently large and $n>K$. Hence we choose $p_0>8K^2$ s.t. $[\sqrt{2p_0}/4]+1$ is prime and set $n=[\sqrt{2p_0}/4]$.

For each $r\gq1$, let $\omega_r:=(\rho,2^{-r}\alpha,2^{-r}\beta,2^{-r}\gamma,2^{-2r}a,2^{-2r}b)$ and consider the scaled phase function $\Phi_{p_0}(u_0;\omega_r)=2^{-3r}\Phi_{p_0}(2^ru_0;\omega_0)$ for all $u_0\in A'_0,~\omega_0\in\mathbb{S}^{d-1}$. This is again a cubic polynomial with bounded coefficients and so $|\Phi_{p_0}(\cdot;\omega_r)|_{K,A'_0}\lq C_{q,p_0,K}$ for any $K>0$. Scaling each annulus $A'_r$ down to $A'_0\subset B(0,2)$ one has that
\begin{equation*}
T_r(\xi)=\int_{A_0'}e^{i2^{3r}|\xi|\Phi_{p_0}(u_0;\omega_r)}\chi_r(u_0)\td u_0,
\end{equation*}
where $\chi_r(u_0)=2^{2qp_0r}\varphi_0(2^ru_0)\zeta_1(2u_0)/\sigma(2^ru_0)\in C_0^\infty(A'_0)$. Applying Lemma \ref{stationary_phase} again to this new expression of $T_r$ on $A'_0$ one sees that $\forall\delta_r,K>0$,
\begin{equation*}
|T_r(\xi)|\lq C_{q,p_0,K}|\chi_r|_{K,A'_0}\delta_r^{-2K}(2^{3r}|\xi|)^{-K}+2|\chi_r|_{0,A'_0}\Lambda^{2qp_0}(\wt\Gamma_{\delta_r}),
\end{equation*}
where $\wt\Gamma_{\delta_r}=\{v_0\in A'_0:|\nabla\Phi_{p_0}(v_0;\omega_r)|\lq\delta_r\}$. Splitting $\wt\Gamma_{\delta_r}$ according to the set $\wt G_{n,\eta_r}:=G_{n,\eta_r}(\nabla\Phi_{p_0}(\cdot;\omega_r))$ one obtains the estimate \eqref{est_excep} for the measure of the set $\wt\Gamma_{\delta_r}$ again from Lemma \ref{factorisation}, \ref{rho_big} and \ref{rho_small}. Note that here instead of unit frequency we have $1>|\omega_r|\gq2^{-2r}$, and so applying Lemma \ref{rho_small} w.r.t. $\omega_r$ one may choose $\delta_r=2^{-2r}\delta_0,\eta_r=2^{-2r}\eta_0,~\eps_r=2^{-2rn}\eps_0$, with the same values of $n$ and $p_0$, so that for $|\xi|$ sufficiently large,
\begin{equation*}
|\xi|^{K/16}|T_r(\xi)|\lesssim_{q,p_0,n,K}|\chi_r|_{K,A'_0}2^{rK}+|\chi_r|_{0,A'_0}2^{4rn(n-1)}\lesssim_{q,p_0,K}|\chi_r|_{K,A'_0}2^{rp_0/2}.
\end{equation*}
Notice that $|u_0|\gq1/4$ for any $u_0\in A'_0$, and so for any multi-indices $\theta,\tau\in\NN^{2qp_0}$ and $r\gq1$,
\begin{equation*}
2^{r|\theta|}\sup_{u_0\in A'_0}|\partial^\tau\varphi_0(2^ru_0)|\lq4^{|\theta|}\sup_{u_0\in A'_0}\lv(2^ru_0)^\theta\partial^\tau\varphi_0(2^ru_0)\rv\lq4^{|\theta|}\|\varphi_0\|_{|\theta|,|\tau|}.
\end{equation*}
Therefore for any $k,l\gq0$, differentiating $\chi_r$ up to $k$ times by Leibniz's rule one sees that $2^{rl}|\chi_r|_{k,A'_0}\lesssim_{q,p_0,k,l}\|\varphi_0\|_{2qp_0+l+k,k}$ from the boundedness of the derivatives of $\zeta_1$ and $\sigma$. This in turn implies that $|T_r(\xi)|\lesssim_{q,p_0,K}2^{-r}|\xi|^{-K/16}\|\varphi_0\|_{\varkappa,K}$, where $\varkappa=2qp_0+p_0/2+K+1$. Summing up in $r$ and one achieves the bound
\begin{equation*}
|I_{p_0}(\xi)|\lq C_{q,p_0,K}|\xi|^{-K/16}\|\varphi_0\|_{\varkappa,K}.
\end{equation*}
It remains to bound the original integral $I_p(\xi)$ in question for all $p\gq p_0$.

Write $v'_p:=\{(x_{jr},y_{ks}):~j,k=1,\cdots,q;~r,s=p_0+1,\cdots,p\}$, then conditional on $v'_p$ the integral $I_p(\xi)$ can be written as, by the factorisation assumption,
\begin{align}
I_p(\xi)&=\int_{\RR^{2qp'}}\varphi_1(v')\td v'\int_{\RR^{2qp_0}}e^{i|\xi|\Phi_p(v_0;v',\omega_0)}\varphi_0(v_0)\td v_0\nonumber\\
&=:\int_{\RR^{2qp'}}J_p(\xi,v')\varphi_1(v')\td v'.\label{I_p}
\end{align}
If one can show that $|J_p(\xi,v')|$ has a global decay in $|\xi|$ uniformly in $p$ and at most polynomial growth in $v'$, then such a decay should be passed on to $|I_p(\xi)|$ by the finite moments of $\varphi_1$. The idea is that for a fixed value of $v'$ (equivalently, conditional on the random variable $v'_p$) the oscillatory integral $J_p(\xi,v')$ has the same global behaviour in $\xi$ as $I_{p_0}(\xi)$.

Using the same cut-off arguments, it suffices to focus on the case where the amplitude $\varphi_0$ is compactly supported on $B(0,2)\subset\RR^{2qp_0}$. Conditional on $v_p'$ (fixing $v'$), Lemma \ref{stationary_phase} can be readily applied w.r.t. $v_0$ for any $K>0$ and the given choice of $\delta_0$,
\begin{equation}\label{J_bound}
|J_p(\xi,v')|\lq C_{q,p_0,K}|\varphi_0|_{K,B(0,2)}|\Phi_p(\cdot;v')|_{K,B(0,2)}^K\delta_0^{-2K}|\xi|^{-K}+2|\varphi_0|_{0,B(0,2)}\Lambda^{2qp_0}(\Gamma'_{\delta_0}),
\end{equation}
where $\Gamma'_{\delta_0}=\{v_0\in B(0,2):|\nabla\Phi_p(v_0;v')|\lq\delta_0\}$. To estimate the measure of the exceptional set $\Gamma'_{\delta_0}$, which may depend on $v'$, one divides it by the set $G'_{n,\eta_0}:=G_{n,\eta_0}(\nabla\Phi_p(\cdot;v'))$ just like before. The key is then to recognise the $v_0$-derivatives of $\Phi_p(v;\omega_0)$.

Separating the monomials that involve $v_0$ in the phase function $\Phi_p(v;\omega_0)$, we write
\begin{equation}\label{phase_p}
\Phi_p(v;\omega_0)=\Phi_{p_0}(v_0;\omega_0)+\Theta_p(v_0;v',\gamma,\rho)+\Upsilon_p(v';\omega_0),
\end{equation}
where the second term is given by, omitting the summation signs over the repeated indices $(j,k)$ and $(j,k,l)$ like before,
\begin{align*}
\Theta_p(v_0;v',\gamma,\rho)=&\gamma_{jk}\lb\sum_{\substack{r\lq p_0\\p_0<s\lq p}}+\sum_{\substack{s\lq p_0\\p_0<r\lq p}}\rb\frac{1}{r^2-s^2}\lb\frac{r}{s}x_{jr}x_{ks}-y_{jr}y_{ks}\rb\\
&+\rho_{jkl}\lb\sum_{\substack{r\lq p_0\\p_0<s\lq p-r}}+\sum_{\substack{s\lq p_0\\p_0<r\lq p-s}}+\sum_{\substack{r\lq p_0\\p_0-r<s\lq p_0}}\rb\\
&\qquad\quad\left\{-\frac{1}{r(r+s)}\left[(x_{jr}y_{ks}+y_{jr}x_{ks})x_{l,r+s}+(-x_{jr}x_{ks}+y_{jr}y_{ks})y_{l,r+s}\right]\right.\\
&\qquad\qquad+\frac{1}{rs}\left[(x_{jr}y_{ls}+y_{jr}x_{ls})x_{k,r+s}+(-x_{jr}x_{ls}+y_{jr}y_{ls})y_{k,r+s}\right]\\
&\qquad\qquad\left.+\frac{1}{s(r+s)}\left[(-x_{kr}y_{ls}+y_{kr}x_{ls})x_{j,r+s}+(x_{kr}x_{ls}+y_{kr}y_{ls})y_{j,r+s}\right]\right\},
\end{align*}
and $\Upsilon_p(v';\omega_0)$ is the sum of remaining monomials that do not involve $v_0$. Then it is clear that the function $\Phi_p(v;\omega_0)$ has the same derivatives in $v_0$ as the function
\begin{equation*}
\Psi_p(v_0;v',\omega_0):=\Phi_{p_0}(v_0;\omega_0)+\Theta_p(v_0;v',\gamma,\rho).
\end{equation*} 
An important observation is that the polynomial $\Theta_p(v_0;v',\gamma,\rho)$ is \textit{at most quadratic} in $v_0$. The summations over the indices $(r,s)$ are plotted in the shaded areas in Figure \ref{shade1} and \ref{shade2}, corresponding to $\nu^{(p)}-\nu^{(p_0)}$ and $\Delta^{(p)}-\Delta^{(p_0)}$, respectively.

\begin{figure}
\centering
\begin{subfigure}{0.45\textwidth}
\begin{tikzpicture}
\fill[fill=gray] (0,1)--(1,1)--(1,2.5)--(0,2.5);
\fill[fill=gray] (1,0)--(2.5,0)--(2.5,1)--(1,1);
\draw[->] (0,0)--(3,0);
\draw[->] (0,0)--(0,3);
\draw[dotted] (0,0)--(3,3);
\draw (0,1)--(3,1);
\draw (1,0)--(1,3);
\draw (0,2.5)--(3,2.5);
\draw (2.5,0)--(2.5,3);
\filldraw (3,0) node[anchor=west]{$r$};
\filldraw (0,3) node[anchor=south]{$s$};
\filldraw (1,0) node[anchor=north]{$p_0$};
\filldraw (0,1) node[anchor=east]{$p_0$};
\filldraw (2.5,0) node[anchor=north]{$p$};
\filldraw (0,2.5) node[anchor=east]{$p$};
\end{tikzpicture}
\caption{Grey: bilinear in $v_0$ and $v'$.}\label{shade1}
\end{subfigure}
\qquad
\begin{subfigure}{0.45\textwidth}
\begin{tikzpicture}
\fill[fill=gray] (0,1)--(0,2.5)--(1,1.5)--(1,1);
\fill[fill=gray] (1,0)--(2.5,0)--(1.5,1)--(1,1);
\fill[fill=black] (1,0)--(1,1)--(0,1);
\draw[->] (0,0)--(3,0);
\draw[->] (0,0)--(0,3);
\draw (0,1)--(3,1);
\draw (1,0)--(1,3);
\draw (0,1)--(1,0);
\draw (0,2.5)--(2.5,0);
\filldraw (3,0) node[anchor=west]{$r$};
\filldraw (0,3) node[anchor=south]{$s$};
\filldraw (1,0) node[anchor=north]{$p_0$};
\filldraw (0,1) node[anchor=east]{$p_0$};
\filldraw (2.5,0) node[anchor=north]{$p$};
\filldraw (0,2.5) node[anchor=east]{$p$};
\end{tikzpicture}
\caption{Grey: linear in $v_0$ and quadratic in $v'$. Black: quadratic in $v_0$ and linear in $v'$.}\label{shade2}
\end{subfigure}
\caption{Monomials of $\Theta_p(v_0;v',\gamma,\rho)$ in the shaded areas.}
\end{figure}
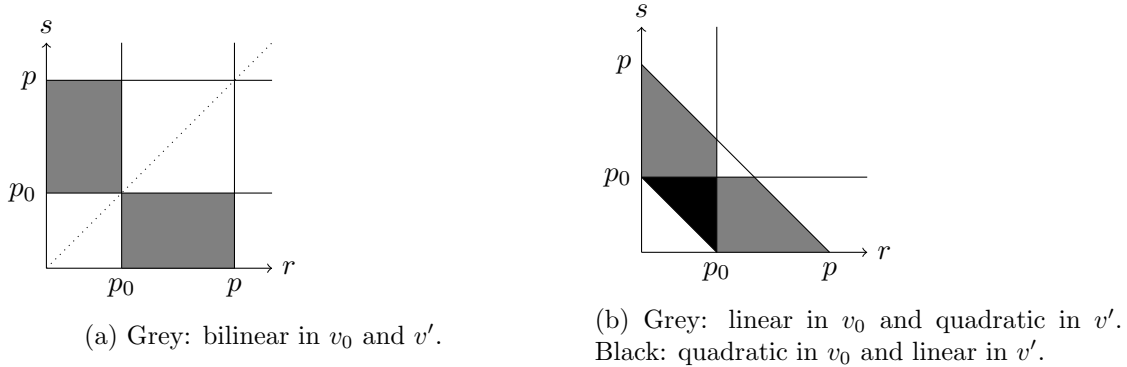

Considering the variable $v_0$ only, the Lipschitz constant of the linear function $\tD^2\Psi_p(v_0;v')$ is identical to that of $\tD^2\Phi_{p_0}(v_0)$ since they only differ by a \textit{constant} matrix $\tD^2\Theta_p(v_0;v',\gamma,\rho)$. Hence Lemma \ref{factorisation} applies directly and gives $\Lambda^{2qp_0}(\Gamma'_{\delta_0}\cap G'_{n,\eta_0})\lq C_{q,p_0}\eta_0^{-2n}\delta_0^n$ uniformly in $p$. This uniformity also holds when applying Lemma \ref{rho_big}: the difference here is that, in its proof, a constant vector $\theta_a(v')$ from the matrix $\tD^2\Theta_p$ is added to each row $q_a(y)$ of the submatrix $Q_m(y)$ therein, and the sets $F_a$ in \eqref{set_diff} are replaced by
\begin{equation*}
F'_a=\{(x,y):~\dist(q_a(y)+\theta_a(v'),~\spa\{q_b(y)+\theta_b(v'):~a\neq b=1,\cdots,n\})>\sqrt{n}\eta_0\}.
\end{equation*}
Then under the image of the same matrix $U_a$, each $F'_a$ is just a translated copy of $F_a$, whose Lebesgue measure remains unchanged. Therefore $\Lambda^{2qp_0}(\Gamma'_{\delta_0}\setminus G'_{n,\eta_0})\lq C_{q,p_0,n}(\eps'_0)^{1-2n}\eta_0^n$ for $|\rho|>\eps'_0$ for any $\eps'_0\in(0,1)$ by Lemma \ref{rho_big}.

However, in order to handle the case where $|\rho|\lq\eps'_0$ the choice of $\eps'_0$ will depend on $v'$, because previously the choice of $\eps_0$ was obtained by studying the specific form of the derivatives of the phase function $\Phi_{p_0}(v_0)$ in Lemma \ref{rho_small}, and now the new phase function $\Psi_p(v_0;v',\omega_0)$ does contain the extra parameter $v'$ in its $v_0$-derivatives. Thankfully, only a slight change (introducing the parameter $v'$) of the proof is needed for Lemma \ref{rho_small} to hold for $\Psi_p$.

For the case where the coefficients $(a,b)$ are dominant, since the first $v_0$-derivatives of $\Theta_p$ are only linear (in $v_0$), the $v_0$-quadratic part of $\nabla_{v_0}\Psi_p$ is the same as that of $\nabla\Phi_{p_0}(v_0;\omega_0)$, and so the estimate \eqref{1st_der_quadratic} still holds for $\eps=\eps_0$. For a fixed value of $v'$, the $v_0$-linear part of $\nabla\Psi_p(v_0;v',\omega_0)$ equals that of $\nabla\Phi_{p_0}(v_0;\omega_0)$ plus $\nabla\Theta_p(v_0;v',\gamma,\rho)$. One then sees that, instead of \eqref{1st_der_linear}, we have a bound $C_{q,p_0}(1+\theta'\sup_{v_0\in B(0,2)}|\nabla_{v_0}\Theta_p|)$ for the linear part, and we choose $\theta'=\theta/(1+\sup_{v_0\in B(0,2)}|\nabla_{v_0}\Theta_p|)$ with the original choice of $\theta$ in the proof.

In the other case where the coefficients $(\alpha,\beta,\gamma)$ are dominant, the constant matrix $A_{p_0}$ in \eqref{2nd_der_mat} is perturbed by the `constant' matrix $\tD^2\Theta_p(v_0;v',\gamma,\rho)$. Instead of looking for a non-singular $n\times n$ submatrix $H_n(v_0;\omega_0)$ of $\tD^2\Phi_{p_0}(v_0;\omega_0)$ one looks for such a submatrix $H_n(v_0;v',\omega_0)$ of $\tD^2_{v_0}\Psi_p(v_0,v';\omega_0)$. From the expression
\begin{equation*}
H_n(v_0;v',\omega_0)=A_{p_0}+\tD^2\Theta_p(v_0;v',\rho)+L_n(v_0;\rho)
\end{equation*}
with the linear part $L_n$ still satisfying \eqref{Ln_bound}, we arrive at the choice $\eps'_0=\eps_0\theta'/(1+|\tD^2_{v_0}\Theta_p|)\lq\eps_0/(1+|\Theta_p(\cdot;v',\gamma,\rho)|^2_{2,B(0,2)})$, as the rest of the arguments in the proof of Lemma \ref{rho_small} remain the same.

Therefore, recalling the choices for $\eps_0,\delta_0,\eta_0$ altogether we have that, for $|\xi|$ sufficiently large,
\begin{align*}
\Lambda^{2qp_0}(\Gamma'_{\delta'})\lesssim_{q,p_0}&\eta_0^{-2n}\delta_0^n+\eps_0^{1-2n}(1+|\Theta_p(\cdot;v',\gamma,\rho)|^2_{2,B(0,2)})^{2n-1}\eta_0^n\\
\lesssim_{q,p_0,n}&|\xi|^{-n/16}\lb1+|\Theta_p(\cdot;v',\gamma,\rho)|^{4n-2}_{2,B(0,2)}\rb.
\end{align*}
Returning to \eqref{J_bound}, with the same values of $n=[\sqrt{2p_0}/4]>K$ and $p_0$ as before we have that, for any $p\gq p_0$ and $\xi$ sufficiently large,
\begin{equation*}
|J_p(\xi,v')|\lesssim_{q,p_0,K}\|\varphi_0\|_{0,K}\lb1+|\Phi_p(\cdot;v',\omega_0)|_{K,B(0,2)}^K+|\Theta_p(\cdot;v',\gamma,\rho)|^{\sqrt{2p_0}-2}_{2,B(0,2)}\rb|\xi|^{-K/16}.
\end{equation*}
Now recalling the separation \eqref{separate_phase} of the $v_0$-monomials in $\Phi_p$ and the fact that the latter is a cubic polynomial, this bound is reduced to $|J_p(\xi,v')|\lesssim_{q,p_0,K}\|\varphi_0\|_{0,K}(1+\sum_\beta|P_\beta(v')|^{\sqrt{2p_0}-2})$. Thus, by the same scaling argument for a general Schwartz function $\varphi_0$ we have the above result with $\|\varphi_0\|_{0,K}$ replaced by the norm $\|\varphi_0\|_{\varkappa,K}$, and the desired bound for $|I_p(\xi)|$ follows from the relation \eqref{I_p}.

The function $I_p(\xi)$ is obviously smooth. Since $|\xi|\Phi_p(v;\omega_0)=\xi\cdot V_p(v)$, with a slight abuse of notation $V_p=V_p(v_p)$, we have, for any multi-index $\alpha\in\NN^d,~|\alpha|=k$, that
\begin{equation*}
\partial^\alpha I_p(\xi)=i^k\int_{\RR^{2qp}}e^{i|\xi|\Phi_p(v;\omega_0)}V_p^\alpha(v)\varphi(v)\td v.
\end{equation*}
Similar to the equality \eqref{separate_phase} separate the $v_0$-monomials in $V_p^\alpha(v)$, and one sees that the derivative $\partial^\alpha I_p(\xi)$ is a sum of oscillatory integrals of the same type as $I_p(\xi)$ itself (with Schwarz amplitudes). The result then follows from the observation that the number of terms in \eqref{separate_phase} depends only on $q,p_0,k$ and that
\begin{equation*}
\max_{|\theta|\lq j,|\sigma|\lq l,|\tau|\lq m}\sup_{v_0\in\RR^{2qp_0}}|v_0^{\theta}\partial^\tau(v_0^{\sigma}\varphi_0(v_0))|\lq\|\varphi_0\|_{j+l,m},
\end{equation*}
for any $j,l,m\in\NN$ and multi-indices $\theta,\sigma,\tau\in\NN^{2qp_0}$.
\end{proof}

As we will encounter standard Fourier integrals as well later on in the next section, perhaps it is a good time now to remark that much less restriction on the amplitude is needed to show the rapid decay for such special cases.

\begin{lemma}\label{fourier_bound}
For arbitrary $K\in\NN$ and $\epsilon>0$, let $h:\RR^d\times\RR^d\to\RR$ be sufficiently smooth w.r.t the first variable s.t. $\sup_z\|h(\cdot,z)\|_{d+\epsilon,K}<\infty$. Then it holds that $\forall z\in\RR^d$,
\begin{equation*}
\lv\int_{\RR^d}e^{\pm iz\cdot u}h(u,z)\td u\rv\lq C_{d,\epsilon}\sup_{z\in\RR^d}\|h(\cdot,z)\|_{d+\epsilon,K}|z|^{-K}.
\end{equation*}
\end{lemma}
\begin{proof}
By assumption $h$ vanishes as $|u|\to\infty$. Then, for each fixed $z$ and any $\alpha\in\NN^d$ s.t. $|\alpha|=K$, by integration by parts we have that
\begin{equation*}
|z|^K\lv\int_{\RR^d}e^{\pm iz\cdot u}h(u;z)\td u\rv=\lv\int_{\RR^d}e^{\pm iz\cdot u}\partial_u^\alpha h(u,z)\td u\rv\lq\int_{\RR^d}|\partial_u^\alpha h(u,z)|\td u.
\end{equation*}
Rewrite the right-hand-side integral as a sum of integrals over the unit ball and the annuli $\{u:2^r\lq|u|<2^{r+1}\}_{r\gq0}$, and the result follows.
\end{proof}

Returning to Theorem \ref{global_estimate}, as a special case the characteristic function $\psi_p(\xi)$ of $V_p$ has Gaussian amplitude $\varphi(v)=\phi_p(v)$, which can be factorised as the product of $\varphi_0(v_0)=\phi_{p_0}(v_0)$ and $\varphi_1(v')=\phi_{p'}(v'):=\phi_p(v)/\phi_{p_0}(v_0)=\prod_{j=1}^q\prod_{r=p_0+1}^p\phi(x_{jr})\phi(y_{jr})$. As the reader will see later, in this case the integral in \eqref{DI_estimate} against $\varphi_1$ is not only finite but \textit{independent of} $p$ as well. This will follow from the moment estimates in the next section.

\section{Moment Estimates and Density Decay}

Recall the notations $d=2q^2+2q+(q^3-q)/3$ and $v_p=\{(x_{jr},y_{ks}):j,k=1,\cdots,q,r,s=1,\cdots,p\}\in\RR^{2qp}$. It is not quite clear yet how the method described in the introduction for the double integral can be applied to the triple integral case. In fact, the expression \eqref{different_g_h} no longer holds as $g$ is no longer the convolution of $f_p$ and the law of $\wt V_p$ - the latter is not independent of $V_p$. Instead, let $\kappa_y,\chi_y$ be the densities of $\wt V_p$ and $\bar V_p$ conditional on that $V_p=y$, respectively. Then one has that
\begin{equation*}
g(z)=\int_{\RR^d}f_p(z-w)\kappa_{z-w}(w)\td w,~h(z)=\int_{\RR^d}f_p(z-w)\chi_{z-w}(w)\td w,
\end{equation*}
and by \eqref{taylor_density}, for all $z\in\RR^d$ one arrives at
\begin{align*}
|g(z)-h(z)|\lq&C_{d,m}\sum_{|\beta|=0}^{m-1}\lv\int_{\RR^d}\lb w^\beta\kappa_{z-w}(w)-w^\beta\chi_{z-w}(w)\rb\td w\rv\\
&+C_{d,m}\sum_{|\beta|=m}\int_{\RR^d}\lv w^\beta\kappa_{z-w}(w)-w^\beta\chi_{z-w}(w)\rv\td w.
\end{align*}
One then sees the complication of estimating the integrands above, compared to the proof of Theorem 15 in \cite{davie2014Patappstodifequusicou}: in the double integral case, due to the independence between $U_p$ and $\wt U_p$ the first integral above will just be $\ex\wt U_p^\beta-\ex\bar{U}_p^\beta$, which vanishes by assumption, and the rest is of order $O(p^{-m/2})$ by Lemma \ref{moments} below (or Lemma 11 in \cite{davie2014Patappstodifequusicou}). However, here $\int_{\RR^d}w^\beta\kappa_{z-w}\td w$ is not even the conditional moment of $\wt V_p$ due to the appearance of $w$ in the subscript of $\kappa$. One may apply Taylor's theorem about $z$ in the subscript and impose certain smoothness condition on $\chi_a$, but whether $\kappa_a$ is smooth in $a$ is not clear.

Instead of this approach we follow a somewhat more primitive way of deriving the coupling bound via the Fourier inversion formula, for which we need much detailed moment estimates for $\wt V_p$.

\begin{lemma}\label{moments}
For fixed $p_0\in\ZZ^+$ and any $p,N\in \ZZ^+,~N>p\gq2p_0$, define the notation $v'_N:=\{(x_{jr},y_{jr}):j=1,\cdots,q;r=p_0+1,\cdots,N\}$ and let $\wt V_{p,N}=V_N-V_p=\wt V_p-\wt V_N$. Then for any $2\lq m\in\ZZ^+$ and $\alpha\in\NN^d$ s.t. $|\alpha|=m$, the following hold:
\begin{description}
\item[(i)] We have that $\wt V_{p,N}^\alpha=\sum_\beta v_{p_0}^\beta\wt P_\beta(v'_N)$, where each $\wt P_\beta$, depending on $p$ and $\alpha$, is a polynomial of degree at most $3m-|\beta|$ and the summation depends on $q,p_0,\alpha$ with $|\beta|\lq m$;
\item[(ii)] For each $\beta$ we have that $\ex|\wt P_\beta(v'_N)|\lesssim_{q,\alpha}p^{-m/2}$ uniformly in $N$.
\end{description}
\end{lemma}
\begin{proof}
One can write $\alpha=(\alpha_1,\alpha_2,\alpha_3,\alpha_4,\alpha_5,\alpha_6)$ s.t. $\sum_{i=1}^6|\alpha_i|=m$ and
\begin{equation}\label{power_split}
\wt V_{p,N}^\alpha=\lb\wt z^{(p,N)}\rb^{\alpha_1}\lb\wt u^{(p,N)}\rb^{\alpha_2}\lb\wt\lambda^{(p,N)}\rb^{\alpha_3} \lb\wt\mu^{(p,N)}\rb^{\alpha_4}\lb\wt\nu^{(p,N)}\rb^{\alpha_5} \lb\wt\Delta^{(p,N)}\rb^{\alpha_6},
\end{equation}
where the terms on the right-hand side are similarly defined as the components of $\wt V_p$ and each multi-index $\alpha_i$ is of corresponding dimension. It is then easier to work with powers of each component.

Clearly the contribution of $v_{p_0}$ comes from $\wt\nu^{(p,N)}$ and $\wt\Delta^{(p,N)}$ only, as the rest are all independent of $v_p$ (and thus of $v_{p_0}$). For each admissible $j,k,l$, the truncated sum $\wt\nu_{jk}^{(p,N)}$ of $\wt\nu_{jk}^{(p)}$ over $p<r\vee s\lq N,~r\neq s$ is at most linear in $v_{p_0}$, and the truncated sum $\wt\Delta_{jkl}^{(p,N)}$ of $\wt\Delta_{jkl}^{(p)}$ over $p<r+s\lq N$ is also at most linear in $v_{p_0}$ as the assumption $p\gq2p_0$ implies that at least one of $r$ and $s$ must be greater than $p_0$. Thus by multiplying out the powers in \eqref{power_split} and re-grouping the monomials involving $v_{p_0}$ one see that the power $\wt V_{p,N}^\alpha$ has degree at most $|\alpha_5|+|\alpha_6|$ in $v_{p_0}$, and the representation (i) follows.

To give a bound for each $\ex|\wt P_\beta(v'_N)|=\ex(|\wt P_\beta(v'_N)||v_{p_0})$, observe that, as per \eqref{power_split}, each $\wt P_\beta(v'_N)$ is a mixed product of the random variables $\wt z^{(p,N)},\wt u^{(p,N)},\wt\lambda^{(p,N)},\wt\mu^{(p,N)},\wt\nu^{(p,N)},\wt\Delta^{(p,N)}$. Separating them by Young's inequality, it suffices to bound the $m$-th moment of each component of $\wt V_{p,N}$. In particular, it suffices to bound the $m$-th conditional (on $v_{p_0}$) moments of $\wt\nu^{(p,N)}$ and $\wt\Delta^{(p,N)}$, since one can then evaluate $v_{p_0}=(1,\cdots,1)$, for example.

It is easy to estimate the moments of the random variables $\wt z^{(p,N)},\wt u^{(p,N)},\wt\lambda^{(p,N)},\wt\mu^{(p,N)}$, since they are all sums of independent random variables. For $\wt u^{(p,N)}=\wt u^{(p)}-\wt u^{(N)}$, each component $\wt u_j^{(p,N)}$ follows $\N(0,\sum_{r=p+1}^Nr^{-4})$ and one sees that $\ex|\wt u^{(p,N)}|^m\lq C_{q,m}p^{-3m/2}$. For $\wt\mu^{(p,N)}$, consider $\wt\mu_{jk}^{(1,p,N)}:=\wt\mu_{jk}^{(1,p)}-\wt\mu_{jk}^{(1,N)}$ for instance: by Rosenthal's inequality (see Theorem 3 in \cite{rosenthal1970SubLppSpabySeqIndRanVar} or Theorem 2.1 in \cite{deacosta1981IneB-vRanVecApptoStrLawLarNum}), for any $N>p$,
\begin{align*}
\ex\lv\sum_{r=p+1}^N\frac{1}{r^2}x_{jr}x_{kr}\rv^m\lesssim_m&\sum_{r=p+1}^N\frac{1}{r^{2m}}\ex|x_{jr}|^m\ex|x_{kr}|^m+\lb\sum_{r=p+1}^N\frac{1}{r^4}\ex|x_{jr}|^2\ex|x_{kr}|^2\rb^{m/2}\\
\lesssim_m&(p+1)^{1-2m}+(p+1)^{-3m/2}.
\end{align*}
Obviously the same bound holds for the $m$-th moment of $\wt{\mu}_{jk}^{(2,p,N)}$, too. One also sees a bound $C_{q,m}p^{-m/2}$ for the $m$-th moments of $\wt z^{(p,N)}$ and $\wt\lambda^{(p,N)}$ for the same reason, but this is in fact implied by part (2) of Lemma 11 in \cite{davie2014Patappstodifequusicou} where a stronger estimate is given. 

It is much less straightforward to compute the conditional moments of $\wt\nu^{(p,N)}$ and $\wt\Delta^{(p,N)}$ as they are not sums of independent random variables. For the rest of the proof use the shorthand notation $\ex_{p_0}:=\ex(\cdot|v_{p_0})$. For each pair $(j,k)$ write $\wt\nu_{jk}^{(p,N)}=A_{jk}-B_{jk}$ where $A_{jk}$ is the corresponding sum of $(r^2-s^2)^{-1}rs^{-1}x_{jr}x_{ks}$ and $B_{jk}$ of $(r^2-s^2)^{-1}y_{jr}y_{ks}$. One can further write (see Figure \ref{split1} below)
\begin{equation*}
A_{jk}=\lb\sum_{\substack{s\lq p_0\\p<r\lq N}}+\sum_{\substack{r\lq p_0\\p<s\lq N}}+\sum_{\substack{p_0<s<r\\p<r\lq N}}+\sum_{\substack{p_0<r<s\\p<s\lq N}}\rb\frac{1}{r^2-s^2}\frac{r}{s}x_{jr}x_{ks}=:T_1-T_2+T_3-T_4,
\end{equation*}
and $B_{jk}$ can be similarly split into four smaller sums. Hence it suffices to bound the $m$-th conditional moment of each of those smaller sums. Moreover, it suffices to consider the case where $m$ is even, as the odd moments can be derived from the Cauchy-Schwartz inequality. 

Multiplying out the power one obtains that
\begin{equation*}
T_1^m=\sum_{\substack{s.\lq p_0\\p<r.\lq N}}\frac{1}{(r_1^2-s_1^2)\cdots(r_m^2-s_m^2)}\frac{r_1\cdots r_m}{s_1\cdots s_m}x_{jr_1}\cdots x_{jr_m}x_{ks_1}\cdots x_{ks_m}
\end{equation*}
and similar expressions for $T_2^m,T_3^m$ and $T_4^m$, where the summations are in $s_\alpha,r_\alpha$ accordingly for all $\alpha=1,\cdots,m$. Note that the random variables $x_{jr_1},\cdots,x_{jr_m}$ are independent of $x_{ks_1},\cdots,x_{ks_m}$ as $j<k$. For the conditional expectation not to vanish, the indices $r_1,\cdots,r_m$ must match in pairs; meanwhile since $p\gq2p_0$, one has that $r_\alpha>2s_\alpha$ and $r_\alpha/(r_\alpha-s_\alpha)\lq2$ for each $\alpha$. Thus
\begin{equation*}
\ex_{p_0}T_1^m\lesssim_m\lb\sum_{p<r\lq N}\frac{1}{r^2}\rb^{m/2}\lb\sum_{s\lq p_0}\frac{1}{s}x_{ks}\rb^m,
\end{equation*}
and by symmetry $\ex_{p_0}T_2^m$ has the same bound with $k$ replaced by $j$.

As for $T_3$ the indices $s_1,\cdots,s_m$ must also match in pairs. If $r_\alpha>2s_\alpha$ then one immediately obtains a bound $C_mp^{-m/2}$; if $r_\alpha\lq2s_\alpha$, then the corresponding expected sum is bounded by (up to the number of matchings)
\begin{equation*}
\lb\sum_{r>p}\sum_{s<r}\frac{1}{(r-s)^2s^2}\rb^{m/2}\lesssim_m\lb\sum_{r>p}\frac{1}{r^2}\sum_{s<r/2}\frac{1}{s^2}\rb^{m/2}\lesssim_mp^{-m/2}.
\end{equation*}
Thus $\ex_{p_0}T_3^m\lq C_mp^{m/2}$, and the same holds for $\ex_{p_0}T_4^m$ by symmetry. So altogether we have that $\ex_{p_0}|A_{jk}|^m\lq C_m(1+(\sum_{r\lq p_0}r^{-1}|x_{jr}|)^m+(\sum_{r\lq p_0}r^{-1}|x_{kr}|)^m)p^{-m/2}$, and it is easy to see that the same bound with $x$ replaced by $y$ holds for $\ex_{p_0}|B_{jk}|^m$.

\begin{figure}
\centering
\begin{subfigure}{0.4\textwidth}
\begin{tikzpicture}
\fill[fill=gray] (0,0)--(2,0)--(2,2)--(0,2);
\draw[->] (0,0)--(4,0);
\draw[->] (0,0)--(0,4);
\draw (0,1)--(4,1);
\draw (1,0)--(1,4);
\draw (0,2)--(2,2);
\draw (2,0)--(2,2);
\draw (2,2)--(4,4);
\filldraw (4,0) node[anchor=west]{$r$};
\filldraw (0,4) node[anchor=south]{$s$};
\filldraw (1,0) node[anchor=north]{$p_0$};
\filldraw (0,1) node[anchor=east]{$p_0$};
\filldraw (2,0) node[anchor=north]{$p$};
\filldraw (0,2) node[anchor=east]{$p$};
\filldraw (3,0.5) node{$T_1$};
\filldraw (0.5,3) node{$T_2$};
\filldraw (3,2) node{$T_3$};
\filldraw (2,3) node{$T_4$};
\end{tikzpicture}
\caption{$\nu_{jk}$}\label{split1}
\end{subfigure}
\begin{subfigure}{0.4\textwidth}
\begin{tikzpicture}
\fill[fill=gray] (0,0)--(2,0)--(0,2);
\draw[->] (0,0)--(4,0);
\draw[->] (0,0)--(0,4);
\draw[dotted] (0,0.5)--(4,0.5);
\draw[dotted] (0.5,0)--(0.5,4);
\draw (0,2)--(2,2);
\draw (2,0)--(2,4);
\draw (2,0)--(0,2);
\filldraw (4,0) node[anchor=west]{$r$};
\filldraw (0,4) node[anchor=south]{$s$};
\filldraw (0.5,0) node[anchor=north]{$p_0$};
\filldraw (0,0.5) node[anchor=east]{$p_0$};
\filldraw (2,0) node[anchor=north]{$p$};
\filldraw (0,2) node[anchor=east]{$p$};
\filldraw (1.5,1.5) node{$S_1$};
\filldraw (1,3) node{$S_2$};
\filldraw (3,2) node{$S_3$};
\end{tikzpicture}
\caption{$\Delta_{jkl}$}\label{split2}
\end{subfigure}
\caption{Splitting of the summations outwith the grey areas}
\end{figure}
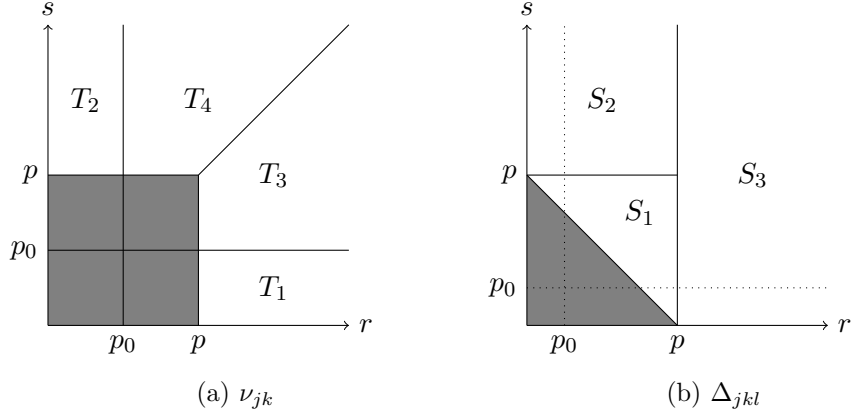

It remains to estimate the conditional moments of $\wt\Delta_{jkl}^{(p,N)}$ for each Lyndon word $(j,k,l)$. Take, for instance, the sum
\begin{align*}
\Sigma_{jkl}:=\sum_{p<r+s\lq N}\frac{1}{rs}x_{jr}y_{ls}x_{k,r+s},
\end{align*}
for a certain Lyndon word $(j,k,l)$. Write $\Sigma_{jkl}=S_1+S_2+S_3$ where, as is illustrated by Figure \ref{split2}, $S_1$ is the sum over $p-s<r\lq p,s\lq p$, $S_2$ is the sum over $p<s\lq N-r,~r\lq p$, and $S_3$ is the sum over $p<r\lq N-s,~s<N-p$. Further write $t=r+s$ for simplicity. Then the $m$-th power of $\Sigma_{jkl}$ can be expressed as
\begin{equation*}
\Sigma_{jkl}^m=\sum_{s.\lq N}\frac{1}{s_1\cdots s_m}y_{ls_1}\cdots y_{ls_m}\lb\sum_{\substack{r.\lq N\\p<t.\lq N}}\frac{1}{r_1\cdots r_m}x_{jr_1}\cdots x_{jr_m}x_{kt_1}\cdots x_{kt_m}\rb,
\end{equation*}
and $S_1,S_2$ and $S_3$ can also be written in this form. We also write $r.$ as the $m$-tuple $(r_1,\cdots,r_m)$ and $s.,t.$ likewise. Denote by $\Pi_m$ the set of all pair-matching patterns for an $m$-tuple.

Notice that in $S_1$ the random variables $x_{jr_1},\cdots,x_{jr_m}$ are all independent of $x_{kt_1},\cdots,x_{kt_m}$. For the conditional expectation not to vanish, the indices $t.$ must match in pairs. Thus
\begin{equation*}
\ex_{p_0}S_1^m=\sum_{s.\lq p}\frac{1}{s_1\cdots s_m}\ex_{p_0}y_{ls_1}\cdots y_{ls_m}\lb\sum_{t.\in \Pi_m}C_{t.}\sum_{r.}\frac{1}{r_1\cdots r_m}\ex_{p_0}x_{jr_1}\cdots x_{jr_m}\rb,
\end{equation*}
where the last summation is over $p-s<r_\alpha\lq p$ for all $\alpha=1,\cdots,m$ subject to a fixed pair-matching pattern of $t.$, and the constant $C_{t.}$ is the product of the corresponding even moments of $x_{kt_\alpha}$. Note that the indices $r.$ and $s.$ cannot be both less than or equal to $p_0$ as $p\gq2p_0$, so they must match in pairs, respectively, too. Hence $\ex_{p_0}S_1^m$ is a polynomial in $x_{j1},\cdots,x_{jp_0},y_{l1},\cdots,y_{lp_0}$ of degree $m$. Distributing out the summation above and using the restriction $p-p_0\gq p/2$ one sees that
\begin{align*}
\ex_{p_0}S_1^m\lesssim_m&\lb\sum_{\substack{p-s<r\lq p\\s\lq p_0}}\frac{1}{s^2r^2}y_{ls}^2\rb^{m/2}+\lb\sum_{\substack{p-s<r\lq p_0\\p-p_0<s\lq p}}\frac{1}{s^2r^2}x_{jr}^2\rb^{m/2}+\lb\sum_{\substack{(p-s)\vee p_0<r\lq p\\p_0<s\lq p}}\frac{1}{s^2r^2}\rb^{m/2}\\
\lesssim_m&\lb\frac{1}{p}\sum_{s\lq p_0}\frac{1}{s^2}(x_{js}^2+y_{ls}^2)\rb^{m/2}+\lb\frac{1}{p}\sum_{s\lq p}\frac{1}{s(p-s+1)}\rb^{m/2}.
\end{align*}
The last summation is bounded by $2\sum_{s\lq p/2}s^{-2}$, and therefore $\ex_{p_0}S_1^m\lq C_m(1+|v_{p_0}|^m)p^{-m/2}$.

For $S_2$, the random variables $x_{jr.}$ and $x_{kt.}$ are still independent as $r.\lq p$. In addition to $t.$, under conditional expectation the indices $s.$ must match in pairs. This means that the indices $r.$ must also match in pairs, and hence
\begin{equation*}
\ex_{p_0}S_2^m=\lb\sum_{p<s.\lq N}\frac{C_{s.}}{s_1^2\cdots s_{m/2}^2}\rb\lb\sum_{r.\lq p_0}\frac{C_{t.}}{r_1^2\cdots r_{m/2}^2}x_{jr_1}^2\cdots x_{jr_{m/2}}^2+\sum_{p<r.\lq N}\frac{C_{t.,r.}}{r_1^2\cdots r_{m/2}^2}\rb,
\end{equation*}
where the constants reflect the number of pair-matching patterns for $s.,t.$ and $r.$ and the even moments of $y_{ls_.},x_{kt.}$ and $x_{jr.}$. This is a polynomial in $x_{j1},\cdots,x_{jp_0}$ of degree $m$, and it has the same bound as $\ex_{p_0}S_1^m$ uniformly in $N$.

In $S_3$ some index $r_\alpha$ may match some other $t_\beta$, thus the we need to consider some more specific cases. Recall that there are only two types of Lyndon words for $3$-tuple $(j,k,l)$. For the case where $j<k\wedge l$, we still have the independence between the random variables $x_{jr_1},\cdots,x_{jr_m}$ and $x_{kt_1},\cdots,x_{kt_m}$. Thus, after taking conditional expectation, only those with pair-matching indices $r.$ and $t.$, respectively, remain non-vanishing. The indices $s.$ must also match in pairs, and therefore similar to $S_2$ one has that
\begin{equation*}
\ex_{p_0}S_3^m=\lb\sum_{p<r.\lq N}\frac{C_{r.,t.}}{r_1^2\cdots r_{m/2}^2}\rb\lb\sum_{s.\lq p_0}\frac{1}{s_1^2\cdots s_{m/2}^2}y_{ls_1}^2\cdots y_{ls_{m/2}}^2+\sum_{p_0<s.\lq N}\frac{C_{s.}}{s_1^2\cdots s_{m/2}^2}\rb,
\end{equation*}
which again leads to the same bound.

It is more intricate to deal with the case where $j=k<l$, as the independence between $x_{jr_1},\cdots,x_{jr_m}$ and $x_{jt_1},\cdots,x_{jt_m}$ is no longer true. For the conditional expectation not to vanish, the $2m$-tuple $\tau:=(r_1,\cdots,r_m,t_1,\cdots t_m)$ must match in pairs. So it suffices to consider the following terms in $S_3$:
\begin{equation}\label{S_3}
\sum_{s.\lq N}\frac{1}{s_1\cdots s_m}y_{ls_1}\cdots y_{ls_m}\lb\sum_{\tau\in \Pi_{2m}}\sum_{r.}\frac{1}{r_1\cdots r_m}x_{jr_1}\cdots x_{jr_m}x_{jt_1}\cdots x_{jt_m}\rb,
\end{equation}
where the last summation is over $p<r.\lq N$ subject to a pair-matching pattern of $\tau$. For fixed values of $s_1,\cdots,s_m$ and a pattern $\tau$, define $\alpha$ and $\beta$ as equivalent on $\{1,\cdots,m\}$ if $r_\alpha$ or $t_\alpha$ is equal to $r_\beta$ or $t_\beta$. Consider the equivalence relation generated by this relation. If $\alpha$ and $\beta$ are in an equivalent class $E\subset\{1,\cdots,m\}$, then the difference $r_\alpha-r_\beta$ is determined by the fixed choice of $s_\alpha,s_\beta$ and the matching constraint of $r_\alpha,t_\alpha,r_\beta,t_\beta$. In effect, for any $\alpha\in E$ the value of $r_\alpha$ determines the values of $r_\beta$ for all $\beta\in E$. Thus, one can choose $r_{\alpha_\ast}=\min\{r_\alpha:\alpha\in E\}$ and rewrite the last summation above as
\begin{equation*}
\sum_{r.}\frac{1}{r_1\cdots r_m}x_{jr_1}\cdots x_{jr_m}x_{kt_1}\cdots x_{kt_m}=\prod_E\sum_{p<r_{\alpha_\ast}\lq N}\frac{1}{r_{\alpha_1}\cdots r_{\alpha_{|E|}}}x_{jr_{\alpha_1}}\cdots x_{jr_{\alpha_{|E|}}}x_{jt_{\alpha_1}}\cdots x_{jt_{\alpha_{|E|}}}
\end{equation*}
where the product is taken over the equivalent class partition of the set $\{1,\cdots,m\}$. Then the expectation of the terms in the big parentheses in \eqref{S_3} is bounded by
\begin{equation*}
C_m\prod_E\sum_{r_{\alpha_\ast}>p}r_{\alpha_\ast}^{-|E|}\lq C_m\prod_Ep^{1-|E|}=C_mp^{n_m-m},
\end{equation*}
where $n_m$ is the number of equivalent classes, which is at most $m/2$, giving an upper bound $C_mp^{-m/2}$ for the quantity above. Thus $\ex_{p_0}S_3^m$ is again a polynomial in $y_{l1},\cdots,y_{lp_0}$ of degree $m$, and one has that
\begin{equation*}
\ex_{p_0}S_3^m\lesssim_mp^{-m/2}\lb\sum_{s\lq p_0}\frac{1}{s}y_{ls}\rb^m+p^{-m/2}\lb\sum_{p_0<s\lq N}\frac{1}{s^2}\rb^{m/2},
\end{equation*}
which again gives the same bound as previous cases.

It then follows from the triangle inequality that $\ex_{p_0}\Sigma_{jkl}^m\lq C_mp^{-m/2}(1+|v_{p_0}|_\ast^m)$. The same arguments apply to all other terms in $\wt \Delta^{(p)}$ (as for the terms $y_{jr}y_{ls}y_{k,r+s}$, the indices $j,k,l$ are never all the same), and the result is proved.
\end{proof}


\begin{remark}\label{moment2}
From the proof one sees that $\ex(|\wt V_p|^m|v_{p_0})\lesssim_{q,m}(1+|v_{p_0}|_\ast^m)p^{-m/2}$, where $|v_{p_0}|_\ast:=\sum_{j=1}^q\sum_{r\lq p_0}r^{-1}(|x_{jr}|+|y_{jr}|)$. Taking expectation again one gets $\ex|\wt V_p|^m\lq C_{q,m}p^{-m/2}$, and by setting $p_0=1,~p=2$ one obtains from the triangle inequality (obviously $V_2$ has bounded moments) that $\sup_{p\gq1}\ex|V_p|^m\lq C_{q,m}$.
\end{remark}

Returning to the applicaiton of Theorem \ref{global_estimate} to the characteristic function $\psi_p(\xi)$ of $V_p$, observe that the phase function can be written as $\Phi_p(v;\omega_0)=\omega_0\cdot V_p$, and so the polynomials $P_\beta$ in \eqref{separate_phase} correspond to the $\wt P_\beta$ in (i) of Lemma \ref{moments} with $|\alpha|=1,N=p$. Then in this case Lemma \ref{moments} (ii) implies that the integral against $\varphi_1=\phi_{p'}$ in \eqref{DI_estimate} is bounded for all $p\gq2p_0$, and thereby Theorem \ref{global_estimate} immediately implies:

\begin{theorem}\label{density}
The density $f_p$ of $V_p$ has continuous and uniformly bounded derivatives up to order $N$ if $p\gq2p_0$, where $p_0>2048(N+d)^2$ is an even integer s.t. $[\sqrt{2p_0}/4]+1$ is prime.
\end{theorem}

This is an analogue of part (1) of Lemma 11 in \cite{davie2014Patappstodifequusicou}; it is not clear whether part (2) of that lemma holds in the triple integral case. But at least we can conclude that the density $f_p$ converges uniformly in $p$ and we have the following:

\begin{corollary}
The random variable $V$ has a density with continuous and bounded derivatives up to any order.
\end{corollary}

By the mean value theorem for $p>p_0$ all the derivatives of $f_p$ up to order $N$ are Lipschitz. Combining this fact with the last comment in Remark \ref{moment2} one can show the rapid decay of the density $f_p$ of $V_p$ for large $p$, due to the following observation.

\begin{lemma}\label{density_decay}
Let $f:\RR^d\to\RR$ be a Lipschitz function with Lipschitz constant $L$. If the moments $\mu_m:=\int_{\RR^d}|x|^m|f(x)|\td x<\infty$ for all $m>0$, then $|f(x)|\lesssim_{d,L}\mu_{r(d+1)}^{1/(d+1)}|x|^{-r}$ for any $r>0$ and $|x|$ sufficiently large.
\end{lemma}
\begin{proof}
If $f\equiv0$ outside a large ball then the claim is trivial. Otherwise take an $x\notin B(0,|f(0)|/L)$ s.t. $\alpha:=|f(x)|>0$. By the Lipschitz condition $|f(x)|<2L|x|$. Moreover, for any $y\in B(x,\frac{3\alpha}{4L})$ one has that $|f(y)|\gq|f(x)|-L|x-y|\gq\alpha/4$. Thus, by assumption for any $m>0$ we have that
\begin{equation*}
\mu_m\gq\int_{B(x,\frac{3\alpha}{4L})\cap\{|y|\gq|x|\}}|y|^m|f(y)|\td y\gq C_d\lb\frac{3\alpha}{4L}\rb^d|x|^m\frac{\alpha}{4}=C_{d,L}|x|^m\alpha^{d+1},
\end{equation*}
which in turn implies that $|f(x)|\lq C_{d,L}(\mu_m|x|^{-m})^{1/(d+1)}$. For any $r>0$ let $m=r(d+1)$, and the result follows from the arbitrary choice of $x$.
\end{proof}

Henceforth we have that $|f_p(y)|\lq C_{d,p_0,r}|y|^{-r}$ for any $r>0,~p>p_0$ and $y\in\RR^d$ sufficiently far away from $0$. Moreover, the argument for the interpolation inequality \eqref{interpolation} also applies here, with $m$ replace by any $k\gq0$ and the decay rate $C_{q,m}e^{-c_q|y|}$ replaced by $C_{d,p_0,k}|y|^{-r}$ for the derivative $\tD^kf_p$. Therefore we arrive at the following conclusion.

\begin{theorem}\label{density_smooth}
For any $N\gq1$ let $p_0$ be defined as in Theorem \ref{density}. Then for all $p\gq2p_0$, all the derivatives of $f_p$ up to order $N$ have rapid decay.
\end{theorem}

\section{Main Result and Further Discussion}

We are now finally ready to proceed to the coupling result, by which we wish to characterise a candidate random variable $\bar V_p$ s.t. the distance $\wass_2(V,V_p+\bar V_p)$ is small.

\begin{theorem}\label{main}
Let $2\lq m\in\ZZ^+$ and $p_0>2048(m+3d+3)^2$ be an even integer s.t. $[\sqrt{2p_0}/4]+1$ is prime. For any $p\gq2p_0$, suppose there exists an $\RR^d$-random variable $\bar{V}_p$ having the same conditional moments given $v_p$ as those of $\wt V_p$ up to order $m-1$ and having density $\chi_y$ conditional on that $V_p=y$. If the function $y\mapsto\chi_y(w)$ is at least $C^{m+d+1}(\RR^d)$ and for $n,k\gq0$ there is a constant $M(n,k)\gq0$ s.t.
\begin{equation*}
\int_{\RR^d}|w|^n|\tD_y^k\chi_y(w)|\td w\lq C_{d,n,k}(1+|y|^{M(n,k)})p^{-n/2},
\end{equation*}
then there exists a constant $C$ depending on $d,m,p_0$ and the density $f_p$ of $V_p$ s.t. $\forall p\gq2p_0$,
\begin{equation*}
\wass_2(V,V_p+\bar{V}_p)\lq Cp^{-m/4}.
\end{equation*}
\end{theorem}
\begin{proof}
Let $g$ and $h$ be the densities of $V$ and $\bar{V}:=V_p+\bar{V}_p$, respectively. Then by the inversion formula and the tower property, for all $z\in\RR^d$,
\begin{align*}
g(z)-h(z)=&\frac{1}{(2\pi)^d}\int_{\RR^d}e^{-iz\cdot\xi}\lb\ex e^{i\xi\cdot V}-\ex e^{i\xi\cdot\bar{V}}\rb\td\xi\\
=&\frac{1}{(2\pi)^d}\int_{\RR^d}e^{-iz\cdot\xi}\ex\lb e^{i\xi\cdot V_p}\ex_p\lb e^{i\xi\cdot\wt{V}_p}-e^{i\xi\cdot\bar{V}_p}\rb\rb\td\xi,
\end{align*}
where $\ex_p:=\ex(\cdot|v_p)$. Applying Taylor's theorem to $\exp(i\xi\cdot\wt V_p)$ and $\exp(i\xi\cdot\bar V_p)$ inside the conditional expectation up to order $m$, one sees that the first $m-1$ differences vanish due to the moment matching assumption and hence
\begin{equation*}
\ex_p\lb e^{i\xi\cdot\wt{V}_p}-e^{i\xi\cdot\bar{V}_p}\rb=i^m\sum_{|\alpha|=m}\frac{m}{\alpha!}\xi^\alpha\int_0^1(1-\theta)^{m-1}\ex_p\lb\wt V_p^\alpha e^{i\theta\xi\cdot\wt V_p}-\bar V_p^\alpha e^{i\theta\xi\cdot\bar V_p}\rb\td\theta.
\end{equation*}
Thus we have the identity
\begin{equation*}
g(z)-h(z)=\frac{1}{(2\pi)^d}\sum_{|\alpha|=m}i^m\frac{m}{\alpha!}\int_0^1(1-\theta)^{m-1}\int_{\RR^d}e^{-iz\cdot\xi}\xi^\alpha\lb\rho(\xi)-\eta(\xi)\rb\td\xi\td\theta,
\end{equation*}
where
\begin{equation*}
\rho(\xi)=\rho(\xi;p,\alpha,\theta):=\ex\lb e^{i\xi\cdot(V_p+\theta\wt V_p)}\wt V_p^\alpha\rb,
\end{equation*}
and $\eta(\xi)$ is similarly defined by replacing $\wt V_p$ with $\bar V_p$. The goal is then to show the rapid decay in $|z|$ of the $\td\xi$ integral above with an appropriate rate of decay in $p$. This should follow from the rapid decay in $|\xi|$ of the derivatives of $\rho(\xi)$ and $\eta(\xi)$.

To this end it is easier to work with, instead of $\wt V_p$, the `truncated remainder' $\wt V_{p,N}:=V_N-V_p$ up to some integer $N\gg p$. Replace $\wt V_p$ with $\wt V_{p,N}$ in $\rho(\xi)$ and denote it by $\rho_N(\xi)$. Recall the notations $v'_N=\{(x_{jr},y_{jr}):j=1,\cdots,q;r=p_0+1,\cdots,N\}$ and $N'=N-p_0$. For $p\gq2p_0$, recall from Lemma \ref{moments} (i) the power $\wt V_{p,N}^\alpha$, as a function of $v_{p_0}$ and $v'_N$, takes the form
\begin{equation*}
\wt V_{p,N}^\alpha=\sum_\beta v_{p_0}^\beta\wt P_\beta(v'_N),
\end{equation*}
where the number of summands depends only on $p_0$ and $\alpha$. By the tower property again,
\begin{align*}
\rho_N(\xi):=&\ex\lb e^{i\xi\cdot(V_p+\theta\wt V_{p,N})}\wt V_{p,N}^\alpha\rb=\ex\sum_\beta\wt P_\beta(v'_N)\ex\lb\left.e^{i\xi\cdot(V_p+\theta\wt V_{p,N})}v_{p_0}^\beta\rv v'_N\rb\\
=&\sum_\beta\int_{\RR^{2qN'}}\wt P_\beta(v')\phi_{N'}(v')\td v'\int_{\RR^{2qp_0}}e^{i|\xi|\wh\Phi_N(v_0,v';\theta,\omega_0)}v_0^\beta\phi_{p_0}(v_0)\td v_0,
\end{align*}
where $\omega_0=\xi/|\xi|$ and the phase function $\wh\Phi_N(v_0,v';\theta,\omega_0)$ is similar to $\Phi_N(v;\omega_0)$ in the sense that some of the terms in $\wh\Phi_N(v;\theta,\omega_0)-\Phi_{p_0}(v_0)$ have dilated frequency $\theta\omega_0$ - recall the decomposition \eqref{phase_p} with $p$ replaced by $N$.

The functions $\varphi_0(v_0):=v_0^\beta\phi_{p_0}(v_0)$ and $\varphi_1(v'):=\wt P_\beta(v')\phi_{N'}(v')$ are both Schwartz, and for any $K,k\in\NN,~\|\varphi_0\|_{\varkappa+3k,K}\lq C_{q,p_0,k,K}$. Similar to \eqref{separate_phase} write $\wh\Phi_N(v;\theta,\omega_0)=\sum_\gamma v_0^\gamma\wh P_\gamma(v')$. Then Lemma \ref{moments} (ii), together with the Cauchy-Schwartz inequality, gives
\begin{align*}
\int_{\RR^{2qN'}}\lb1+\sum_\gamma|\wh P_\gamma(v')|^{\sqrt{2p_0}+k-2}\rb\varphi_1(v')\td v'=&\ex\lb1+\sum_\gamma|\wh P_\gamma(v'_N)|^{\sqrt{2p_0}+k-2}\rb|\wt P_\beta(v'_N)|\\
\lq&C_{q,p_0,\alpha,k}p^{-m/2},
\end{align*}
for all multi-indices $\beta$ and all $k\in\NN$. Moreover, for fixed $\theta$ and $v'$ the function $\wh\Phi_N(v;\theta,\omega_0)$ only differs from $\Phi_{p_0}(v_0;\omega)$ by a quadratic polynomial in $v_0$ with no singularity in $\theta$, which is insignificant according to the last part of the proof of Theorem \ref{global_estimate} (an additional power in $v'$ may appear in the integral against $\varphi_1$, but this integral will again be independent of $p$ by Lemma \ref{moments}). Thus, by Theorem \ref{global_estimate} (for $K=m+d+1$) the integral $\rho_N(\xi)$ is a smooth function s.t. for all $k\gq0,~N\gg p,~\theta\in(0,1)$ and $|\xi|$ sufficiently large,
\begin{equation}\label{rho_derivatives}
|\tD^k\rho_N(\xi)|\lq C_{q,p_0,m,k}|\xi|^{-m-d-1}p^{-m/2}.
\end{equation}
Thereby for all $N$ and $\theta$ the function $G_N(\xi):=\xi^\alpha\rho_N(\xi)$ has uniformly bounded derivatives, and in particular $\|G_N\|_{d+1,d+3}\lq\|\rho_N\|_{m+d+1,d+3}\lq C_{d,p_0,m}p^{-m/2}$. Applying Lemma \ref{fourier_bound} with ($k\lq$)$K=d+3$ one deduces that
\begin{equation}\label{rho_decay}
\lv\int_{\RR^d}e^{-iz\cdot\xi}\xi^\alpha\rho_N(\xi)\td\xi\rv\lq C_{d,p_0,m}|z|^{-d-3}p^{-m/2}.
\end{equation}
This estimate is uniform in $N$, and therefore by taking the limit $N\to\infty$ the same bound holds for the integral $\int_{\RR^d}e^{-iz\cdot\xi}\xi^\alpha\rho(\xi)\td\xi$.

For the other integral $\int_{\RR^d}e^{-iz\cdot\xi}\xi^\alpha\eta(\xi)\td\xi$, from the same arguments above it suffices to show that $\forall0\lq k\lq d+3,~N\gg p,~\theta\in(0,1)$ and $|\xi|$ sufficiently large the estimate \eqref{rho_derivatives} also holds for $\eta$. By conditioning on the value of $V_p$ one finds the identity
\begin{equation*}
\eta(\xi)=\ex e^{i\xi\cdot V_p}\ex\lb e^{i\theta\xi\cdot\bar V_p}\bar V_p^\alpha\left|V_p\right.\rb=\int_{\RR^{d}}e^{i\xi\cdot y}f_p(y)\int_{\RR^d}e^{i\theta\xi\cdot w}w^\alpha\chi_y(w)\td w\td y.
\end{equation*}
Differentiating $\eta(\xi)$ by $k$ times by Leibniz's rule, one has that $\forall\tau\in\NN^d,~|\tau|=k$,
\begin{align*}
\partial^\tau\eta(\xi)&=\int_{\RR^d}e^{i\xi\cdot y}f_p(y)\sum_{\sigma\lq\tau}\binom{\tau}{\sigma}(iy)^{\tau-\sigma}\partial_\xi^\sigma\lb\int_{\RR^d}e^{i\theta\xi\cdot w}w^\alpha\chi_y(w)\td w\rb\td y\\
&=i^k\sum_{\sigma\lq\tau}\binom{\tau}{\sigma}\theta^{|\sigma|}\int_{\RR^d}e^{i\xi\cdot y}\underbrace{y^{\tau-\sigma}f_p(y)\int_{\RR^d}e^{i\theta\xi\cdot w}w^{\sigma+\alpha}\chi_y(w)\td w}_{=:H_{\tau,\sigma}(y;\xi,\theta)}\td y.
\end{align*}
Then by Leibniz's rule again for any $\beta\in\NN^d,~|\beta|\lq m+d+1$,
\begin{align*}
\partial_y^\beta H_{\tau,\sigma}(y;\xi,\theta)=&\sum_{\nu\lq\beta}\binom{\beta}{\nu}\partial_y^{\beta-\nu}(y^{\tau-\sigma}f_p(y))\int_{\RR^d}e^{i\theta\xi\cdot w}w^{\sigma+\alpha}\partial_y^\nu\chi_y(w)\td w,
\end{align*}
which is bounded in $\xi$ and $\theta$. Thus, by the moment assumption on the conditional density of $\chi_y$ and counting the maximum power of $y$, one sees from Theorem \ref{density_smooth} that
\begin{equation*}
\sup_{\xi\in\RR^d,\theta\in(0,1)}\|H_{\tau,\sigma}(\cdot;\xi,\theta)\|_{d+1,m+d+1}\lq C_{m,d,k}p^{-m/2}\|f_p\|_{k+d+1+M(m+k,m+d+1),m+d+1},
\end{equation*}
for all $\theta\in(0,1)$ and $\xi\in\RR^d$. Thus the sought-after estimate \eqref{rho_derivatives} for $|\tD^k\eta(\xi)|$ follows from Lemma \ref{fourier_bound} for $K=m+d+1$, with the Schwarz norm of $f_p$ above as a multiplicative factor. Apply Lemma \ref{fourier_bound} again for ($k\lq$)$K=d+3$ we obtain \eqref{rho_decay} with $\rho_N$ replaced by $\eta$ and a multiplicative factor $\|f_p\|_{2d+4+M(m+d+3,m+d+1),m+d+1}$.

The result then follows from the inequality \eqref{crude_bound_1} for $p=2$.
\end{proof}

To finish off this article I shall make the following remarks regarding the remaining difficulties of the coupling problem for the triple stochastic integral.

\paragraph{Rate of convergence.} As opposed to the rate $O(p^{-m/2})$ obtained in \cite{davie2014Patappstodifequusicou} for the double integral, the rate $O(p^{-m/4})$ is probably the best one can expect simply from Theorem \ref{global_estimate} and Theorem \ref{density} alone. This is because the particular form of the phase function $\Phi_p$ and its derivatives are not fully exploited. In fact we have only used the fact that the phase function $\Phi_p$ is a cubic polynomial in Lemma \ref{rho_big} and Lemma \ref{rho_small}. Moreoever, the inequality \eqref{crude_bound_1} itself is not very sharp, especially for the quadratic distance $\wass_2$.

Despite this limitation, to my best knowledge what is proved so far is the first attempt to find a coupling for triple stochastic integrals. I believe that Davie's rate $O(p^{-m/2})$ could still be achieved if analogues of Lemma 12, 13 and especially 14 in \cite{davie2014Patappstodifequusicou} can be proved.

\paragraph{Generation of $\bar V_p$.} The requirements for the candidate $\bar V_p$ in Theorem \ref{main} are not straightforward. While the moment-matching condition is relatively easy to meet (similar to the discussion after Theorem 15 in \cite{davie2014Patappstodifequusicou}), it is not clear how to generate $\bar V_p$ s.t. its conditional density satisfies the smoothness conditions.

\paragraph{Application to SDE approximation.} The numerical scheme based on a coupling of order $O(p^{-m/4})$ is computationally equivalent to the Milstein scheme based on Wiktorsson's result \cite{wiktorsson2001JoiChaFunSimSimIteItIntMulIndBroMot} with step size $h^{3/2}$- see the discussion following the proof of Theorem 15 in \cite{davie2014Patappstodifequusicou}. To achieve a genuine improvement one needs a better rate than Theorem \ref{main}, which requires careful estimates for the density $f_p$ of $V_p$ as mentioned above. Moreover, it is also not clear how one can combine Davie's coupling for the double integral and the coupling for the triple integral together to form a genuinedly improved numerical scheme. These questions are left open for future investigation.

\bibliographystyle{acm}
\bibliography{xiling}

\begin{thebibliography}{10}

\bibitem{clark1978MaxRatConDisAppStoDifEqu}
{\sc Clark, J. M.~C., and Cameron, R.~J.}
\newblock The maximum rate of convergence of discrete approximations for
  stochastic differential equations.
\newblock {\em Stochastic differential systems (Proc. IFIP-WG 7/1 Working
  Conf., Vilnius)\/} (1978), 162--171.

\bibitem{davie2014Patappstodifequusicou}
{\sc Davie, A.~M.}
\newblock Pathwise approximation of stochastic differential equations using
  coupling.
\newblock {\em Preprint\/} (2014), www.maths.ed.ac.uk/$\sim$sandy/coum.pdf.

\bibitem{davie2015Polpernordis}
{\sc Davie, A.~M.}
\newblock Polynomial perturbations of normal distributions.
\newblock {\em Preprint\/} (2016), www.maths.ed.ac.uk/$\sim$sandy/polg.pdf.

\bibitem{deacosta1981IneB-vRanVecApptoStrLawLarNum}
{\sc de~Acosta, A.}
\newblock Inequalities for {$B$}-valued random vectors with applications to the
  strong law of large numbers.
\newblock {\em The Annals of Probability 9}, 1 (1981), 157--161.

\bibitem{gaines1994AlgIteStoInt}
{\sc Gaines, J.~G.}
\newblock The algebra of iterated stochastic integrals.
\newblock {\em Stochastics and Stochastics Reports 49\/} (1994), 169--179.

\bibitem{hormander1990AnaLinParDifOpeI:DisTheFouAna}
{\sc H\"{o}rmander, L.}
\newblock {\em {The Analysis of Linear Partial Differential Operators {I}:
  Distribution Theory and {Fourier} Analysis}}, 2nd~ed.
\newblock Springer-Verlag, 1990.

\bibitem{kloeden1995NumSolStoDifEqu}
{\sc Kloeden, P.~E., and Platen, E.}
\newblock {\em {Numerical Solution of Stochastic Differential Equations}}.
\newblock Springer-Verlag, 1995.

\bibitem{lyons1994RanGenStoAreInt}
{\sc Lyons, T.~J., and Gaines, J.~G.}
\newblock Random generation of stochastic area integrals.
\newblock {\em SIAM Journal on Applied Mathematics 54}, 4 (1994), 1132--1146.

\bibitem{rosenthal1970SubLppSpabySeqIndRanVar}
{\sc Rosenthal, H.~P.}
\newblock On the subspaces of $l^p(p>2)$ spanned by sequences of independent
  random variables.
\newblock {\em Israel Journal of Mathematics 8\/} (1970), 273--303.

\bibitem{rudin1976PriMatAna}
{\sc Rudin, W.}
\newblock {\em {Principles of Mathematical Analysis}}, 3rd~ed.
\newblock McGraw-Hill Inc., 1976.

\bibitem{sogge1993FouIntClaAna}
{\sc Sogge, C.~D.}
\newblock {\em {Fourier Integrals in Classical Analysis}}.
\newblock Cambridge Tracts in Mathematics. Cambridge University Press, 1993.

\bibitem{villani2003TopOptTra}
{\sc Villani, C.}
\newblock {\em Topics in Optimal Transportation}, vol.~58.
\newblock American Mathematical Society, 2003.

\bibitem{wiktorsson2001JoiChaFunSimSimIteItIntMulIndBroMot}
{\sc Wiktorsson, M.}
\newblock Joint characteristic function and simultatenous simulation of
  iterated {It\^{o}} integrals for multiple independent brownian motions.
\newblock {\em Annals of Applied Probability 11\/} (2001), 470--487.

\end{thebibliography}
\end{document}